\documentclass{article}
\usepackage{style}
\usepackage[utf8]{inputenc}

\usepackage{xspace,prettyref}
\usepackage{dsfont}
\usepackage{placeins}

\newrefformat{eq}{(\ref{#1})}
\newrefformat{chap}{Chapter~\ref{#1}}
\newrefformat{sec}{Section~\ref{#1}}
\newrefformat{alg}{Algorithm~\ref{#1}}
\newrefformat{fig}{Figure~\ref{#1}}
\newrefformat{tab}{Table~\ref{#1}}
\newrefformat{rmk}{Remark~\ref{#1}}
\newrefformat{clm}{Claim~\ref{#1}}
\newrefformat{def}{Definition~\ref{#1}}
\newrefformat{cor}{Corollary~\ref{#1}}
\newrefformat{lmm}{Lemma~\ref{#1}}
\newrefformat{prop}{Proposition~\ref{#1}}
\newrefformat{prob}{Problem~\ref{#1}}
\newrefformat{app}{Appendix~\ref{#1}}
\newrefformat{hyp}{Hypothesis~\ref{#1}}
\newrefformat{thm}{Theorem~\ref{#1}}

\begin{document}


\title{Finding Planted Cycles in a Random Graph}

 \author{Julia Gaudio\footnote{Northwestern University, Department of Industrial Engineering and Management Sciences; julia.gaudio@northwestern.edu} \and Colin Sandon \footnote{\'Ecole Polytechnique F\'ed\'erale de Lausanne; colin.sandon@epfl.ch} \and Jiaming Xu \footnote{Duke University, The Fuqua School of Business; jiaming.xu868@duke.edu} \and Dana Yang \footnote{Cornell University, Department of Statistics and Data Science; xy374@cornell.edu}}

\date{\today}

\maketitle

\begin{abstract}
In this paper, we study the problem of finding a collection of planted cycles in an \ER random graph $G \sim \mathcal{G}(n, \lambda/n)$, in analogy to the famous Planted Clique Problem. When the cycles are planted on a uniformly random subset of $\delta n$ vertices, we show that almost-exact recovery (that is, recovering all but a vanishing fraction of planted-cycle edges as $n \to \infty$) is information-theoretically possible if $\lambda < \frac{1}{(\sqrt{2 \delta} + \sqrt{1-\delta})^2}$ and impossible if $\lambda > \frac{1}{(\sqrt{2 \delta} + \sqrt{1-\delta})^2}$. 
Moreover, despite the worst-case computational hardness of finding long cycles, we design a polynomial-time algorithm that attains almost exact recovery when $\lambda < \frac{1}{(\sqrt{2 \delta} + \sqrt{1-\delta})^2}$. This stands in stark contrast to the Planted Clique Problem, where a significant computational-statistical gap is widely conjectured. A key technical contribution is a novel generating-function approach for counting imbalanced circuits that arise in decompositions of the symmetric difference between the planted cycles and alternative feasible solutions.
\end{abstract}

\providecommand{\keywords}[1]
{
  \small	
  \textbf{\textit{Keywords---}} #1
}


\newpage 
\section{Introduction}
The problem of finding the longest cycle in a graph $G=([n],E)$ is a classical and fundamental challenge in theoretical computer science. It is well-known to be NP-hard in the worst case. A celebrated color-coding technique can find a cycle of length $k$, if it exists, in  $e^{O(k)}n|E|$ expected time~\cite{Alon1994}. However, this method only gives a polynomial-time algorithm to find a cycle of length $O(\log n). $

Given the problem's computational intractability in worst-case instances, a natural direction is to study it in random graph models, where one can hope for more tractable behavior. This line of inquiry dates back to the foundational work of Erd\H{o}s and R\'enyi~\cite{erdos1960evolution}, and has since grown into a rich area of research.
A celebrated result of \cite{bollobas1984long} shows that in the \ER random graph $\mathcal{G}(n,\lambda/n)$, the length of the longest cycle satisfies $L_{\lambda,n} \geq (1-\lambda^6 e^{-\lambda})n$ for $\lambda > 0$ sufficiently large. Subsequent works~\cite{frieze1986large} have further refined these bounds, culminating in the result of \cite{anastos2021scaling,anastos2023note} proving that $L_{\lambda,n}/n \to f(\lambda)$ almost surely for $\lambda \ge 20$, where $f(\lambda)=1-\sum_{k=1}^\infty p_k(\lambda) e^{-k\lambda}$ and $p_k$ is a polynomial in $\lambda$.
From a computational standpoint, the simple depth-first-search algorithm can find a cycle of length $\Theta(n)$ efficiently~\cite{krivelevich2015long}. However, whether the longest cycle can be found in polynomial time remains an open question.


A similar phenomenon is observed for the problem of finding the largest clique in a graph. It is known that the largest clique in $\mathcal{G}(n,p)$ for constant $p$ is of size $(2 \pm o(1)) \log_{1/p}(n)$ with high probability \cite{Bollobas1998}, yet the best known polynomial-time algorithms can only achieve half of this size \cite{Karp1976}. In response to this gap, a prominent line of research initiated by~\cite{jerrum1992large,kuvcera1995expected} considers the problem of finding a clique which is \emph{planted} inside an Erd\H{o}s--R\'enyi random graph $G \sim \mathcal{G}(n,p)$. Letting $k$ be the size of the planted clique, it is known that one
can recover a planted clique of size $k=\Theta(\log n)$ information-theoretically; however, the best known polynomial-time algorithms 
can only recover a planted clique of size $k = \Omega(\sqrt{n})$ (see e.g. \cite{alon1998finding}).
We study the analogous question for cycles:
\begin{center}
\textit{Under what conditions can we (efficiently) find cycles that are planted in a random graph}?
\end{center}
Concretely, we generate a graph according to the following planted cycles model.\footnote{We note that our planted cycles problem is distinct from the similarly-named \emph{planted dense cycle problem} \cite{mao2023detection,mao2024information}. In the planted dense cycle problem, the vertices of the background graph are spatially embedded on a circle. Additional edges are then added, connecting pairs of vertices randomly with a probability that decays with the circular distance.}
\begin{definition}[Planted cycles model]
To generate a simple graph (i.e., a graph without self-loops or parallel edges) $G$ from the planted cycles model, denoted by $\mathcal{G}(n,  \lambda, \delta)$, first generate a background \ER graph $G_0 = ([n],E) \sim \mathcal{G}(n, \lambda/n)$ by connecting each pair of two vertices independently at random with probability $\lambda/n$. Then choose a set of $\delta n$ vertices\footnote{Strictly speaking we choose $\lfloor \delta n\rfloor$ vertices, but for ease of notation we generally just say $\delta n$.} $V \subseteq [n]$, uniformly at random. Finally, select a labeled $2$-factor 
$H^*$ on $V$, a vertex-disjoint union of cycles spanning all vertices in $V$, uniformly at random.
Let $G = G_0 \cup H^*$ be the union of the background graph and the planted cycles.  
\end{definition}
The goal is to recover the edges in the planted $2$-factor $H^*$ based on the observation $G$. Note that crucially, the vertex set $V$ of the planted $2$-factor $H^*$ is unknown. We focus on the regime where $\lambda, \delta > 0$ are fixed constants, independent of $n$. It turns out that exactly recovering $H^*$ with high probability is information-theoretically impossible for constant $\lambda$ (Theorem \ref{thm:exact}).  We therefore study the problem of almost exact recovery; that is, recovering all but an $o(1)$ fraction of edges in the planted cycles in expectation. More formally, let $\hat{H}$ be an estimator of $H^*$; that is, a set of edges on the complete graph $K_n$. The reconstruction error, namely the fraction of misclassified edges, is 
\begin{equation}
\risk(H^*, \hat H) = \frac{\left| \Mplanted \Delta \hat{H} \right|}{|\Mplanted|},
\label{eq:risk}
\end{equation}
where $\Delta$ denotes the symmetric set difference.
We say that $\hat{H}$ achieves almost exact recovery if $\Expect[\risk(\Mplanted,\hat H)] = o(1)$, or equivalently, 
$\prob{\risk(\Mplanted,\hat H) = o(1)}=1-o(1)$, as $n \to \infty. $

\subsection{Main results}

Our main results precisely characterize the threshold for almost exact recovery of the planted cycles.

\begin{theorem}\label{thm:almost-exact}
Consider the planted cycles model $\calG(n,\lambda,\delta)$, where $\delta \in (0,1]$. If $\lambda < \frac{1}{(\sqrt{2\delta} + \sqrt{1-\delta})^2}$, then almost exact recovery is possible. Moreover, there exists a polynomial-time algorithm that achieves almost exact recovery. Conversely, if $\lambda > \frac{1}{(\sqrt{2\delta} + \sqrt{1-\delta})^2}$, then almost exact recovery is information-theoretically impossible for any algorithm, regardless of computation time. 
\end{theorem}

The case $\delta = 1$ corresponds to a planted $2$-factor on the entire vertex set $[n]$, and was previously studied in  \cite{gaudio2025all}.  In that setting, almost exact recovery is possible if $\lambda \le \frac{1}{2}$, while it is impossible for $\lambda > \frac{1}{2}$. Interestingly, the threshold $\frac{1}{(\sqrt{2\delta}+\sqrt{1-\delta})^2}$ is non-monotonic in $\delta$, attaining its maximum value $1$ at $\delta=0$
and minimum value $1/3$ at $\delta=2/3$. Thus, when $\lambda < \frac{1}{(\sqrt{2\delta} + \sqrt{1-\delta})^2}$, we are always in the subcritical regime of the \ER random graph, which contains only short cycles of bounded length.

A salient challenge in our problem is that the vertex set $V$ of the planted cycles $H^*$ is unknown. If $V$ were known, the problem would reduce to recovering a planted $2$-factor within the induced subgraph $G[V]$, which is distributed as  $\calG(n', \lambda'/n')$ with $n'=\delta n$  and $\lambda'=\delta \lambda$. Thus, in this known-$V$ setting, the sharp threshold for almost exact recovery is $\lambda' = 1/2$, or equivalently $\lambda= 1/(2\delta)$. Since this threshold is always higher than $\frac{1}{(\sqrt{2\delta}+\sqrt{1-\delta})^2}$ for $\delta < 1$, this comparison highlights that the unknown support $V$ renders the recovery task statistically strictly harder.

From a computational perspective, the unknown vertex set $V$ also introduces substantial challenges, as enumerating over all $\binom{n}{\delta n}$ possible subsets requires exponential time.  Nevertheless, and perhaps surprisingly, we establish that the sharp recovery threshold can still be attained  by a polynomial-time algorithm; hence, there is no statistical-computational gap. This stands in stark contrast to the planted clique problem. Information-theoretically, it is possible to recover a planted clique of size  $\Theta(\log n)$. 
However, it is conjectured that no polynomial-time algorithm can find a planted clique of size $o(\sqrt{n})$, and numerous computational hardness results, such as sum-of-squares lower bounds~\cite{barak2019nearly}, have been derived.



\begin{remark}\label{remark:single-cycle}
Note that under the planted cycles model, $H^*$ is a random $2$-factor and thus likely consists of multiple cycles. We can consider a variation of the model where $H^*$ is conditioned to be a single Hamiltonian cycle on $\delta n$ randomly chosen vertices, and our main results continue to hold. In particular, our achievability result in Theorem \ref{thm:almost-exact} holds conditionally on any $H^*$, so in particular it holds when $H^*$ is a single cycle (see Theorem \ref{thm:possibility_almost_exact_partial_2_factor}). Our algorithmic result similarly holds conditionally on $H^*$ (see~\prettyref{thm:alg}). 
As for the impossibility result, we leverage a reduction argument along with the observation that 
a random 2-factor $H^*$ is a cycle with probability at least $1/(\delta n)$. 
The detailed arguments are included in~\prettyref{sec:planted-cycle}, Lemma \ref{lemma:impossibility-cycle}. 
\end{remark}

\subsection{Analytical and algorithmic innovations}

To delineate the achievability and impossibility results, we characterize the structure of the difference graph $H^*\Delta H$ between the planted cycles $H^*$ and any other 2-factor $H$  supported on $\delta n$ vertices. Since $|H^*|=|H|$, exactly half of the edges in $H^*\Delta H$ are planted. Moreover, this difference graph can be decomposed into a vertex-disjoint union of circuits---closed trails\footnote{A trail is a sequence of distinct edges
$(e_1, e_2, \ldots, e_{n-1})$ where $e_i=(v_i,v_{i+1})$ for $1 \le i \le n-1$ and $n\ge 1.$ A circuit is a closed trail with $v_n=v_1.$}---where edges must alternate between planted and unplanted edges at shared vertices in $V(H^*) \cap V(H)$. 

The core challenge arises in analyzing how the presence of these circuits affects the recoverability of $H^*$. When $\delta = 1$, the possibility of recovering $H^*$ is driven by the presence of balanced circuits. Intuitively, if the number of balanced circuits---those with equal numbers of planted and unplanted edges---is small, then any competing 2-factor $H$ must largely overlap with $H^*$, making almost-exact recovery feasible. Conversely, if there exists a diverging number of balanced circuits, then there are many plausible 2-factors that differ significantly from $H^*$, rendering almost-exact recovery impossible. However, beyond the special case $\delta=1$, the situation becomes significantly more complex: individual circuits in the decomposition may be imbalanced, as long as the collection as a whole remains balanced (i.e., has half of its edges planted). This added layer of structural flexibility necessitates more refined tools for counting and characterizing such circuits.


\paragraph{Generating function approach} 
To address this complexity, we develop a novel generating function approach that systematically counts balanced and imbalanced circuits. It turns out that the relevant object to count is the set of \emph{$(a,b)$-trails}—trails consisting of $a$ planted and $b$ unplanted edges, subject to the constraint that no two consecutive unplanted edges meet at a vertex in $V(H^*)$
(See the complete definition in Definition \ref{def:trail}). We show that when the number of $(a,b)$-trails with $b \ge a$ is small, then almost-exact recovery is information-theoretically possible. Conversely, the presence of many perfectly balanced trails (with $a = b$) implies impossibility.

Specifically, let $c_{a,b}$ be a constant such that the expected number of  $(a,b)$-trails from a given vertex to vertices in $V(H^*)$ is approximately equal to $c_{a,b}$.
We show that the generating function of the sequence $
(c_{a,b})_{a,b=1}^{\infty}$ is
\[
g(x,y) \triangleq \sum_{a,b = 1}^{\infty} c_{a,b} x^a y^b = \sum_{k=1}^{\infty} \left(\frac{2x}{1-x} \cdot \frac{\delta \lambda y}{1 - (1-\delta)\lambda y} \right)^k.\]
This generating function admits a natural interpretation. We assign weight $x$ to each planted edge  and weight $y$ to each  unplanted edge, and then enumerate the weighted $(a,b)$-trails from a given vertex to vertices in $V(H^*)$ in expectation as follows:
\begin{itemize} 
\item The factor $\frac{\delta \lambda y}{(1-(1-\delta)\lambda y}=\delta \lambda y \left(1+(1-\delta)\lambda y+((1-\delta)\lambda y)^2+ \ldots \right)$ enumerates an unplanted segment, where the exponent of $y$ records the number of unplanted edges. In such a segment, all but the last edge terminate at unplanted vertices.  Hence, the factor $(1-\delta)\lambda$ captures the expected number of unplanted edges incident to an unplanted vertex, while $\delta \lambda$  is the expected number of unplanted edges incident to a planted vertex.
\item 
    The factor $\frac{2x}{1-x}=2x\left(1+x+x^2+\ldots\right)$  enumerates a planted segment, where the exponent of $x$ records the number of planted edges. The coefficient $2$ arises from the choice of two directions when traversing a planted cycle. 
\item Finally, the summation over $k \geq 1$ accounts for the possibility of switching between unplanted and planted segments arbitrarily many times. 
\end{itemize}

To bound the number of circuits with more unplanted than planted edges, we choose $0 < x < 1 < y$ such that $xy = 1$. In this case, $g(x,y)$ upper bounds $\sum_{a \le b} c_{a,b}$, and more generally, we have $\sum_{ b\ge a+\ell} c_{a,b} \le g(x,y) (x/y)^{\ell/2}$ for any $\ell \ge 0$,\footnote{Note that 
$\sum_{b\ge a+\ell} c_{a,b} = 
\sum_{b \ge a+\ell} c_{a,b} (xy)^{(a+b)/2}
=\sum_{b \ge a+\ell} c_{a,b} x^a y^b (x/y)^{(b-a)/2}
\le (x/y)^{\ell/2} \sum_{b \ge a+\ell} c_{a,b} x^a y^b. 
$
}
which decays exponentially in $\ell.$ Hence, if there exists such a pair $(x,y)$ for which $g(x,y)$ is finite, we can conclude that the total number of $(a,b)$-trails with $b \ge a$ is bounded—implying almost-exact recovery is possible. We show that this occurs precisely when $\lambda < \frac{1}{(\sqrt{2\delta} + \sqrt{1-\delta})^2}$.



Conversely, when $\lambda > \frac{1}{(\sqrt{2\delta} + \sqrt{1-\delta})^2}$, the generating function $g(x,y)$ diverges for all choices of  $x,y$ satisfying $0<x<1<y$ and $xy = 1$. While this divergence does not imply that $c_{a,a}$ is large for every $a$, we prove that there exists an $m^*$ such that  $c_{m^*,m^*}>1 $(cf. \prettyref{lmm:m-star}). This enables us to construct a supercritical branching process that recursively grows many balanced $(a,b)$-paths with $a = b$, rooted at any given vertex. These paths, in turn, can be stitched together into long cycles, ultimately forming exponentially many $2$-factors that significantly deviate from the planted $H^*$. This establishes the impossibility of almost-exact recovery in this regime.


\paragraph{Greedy search algorithm}
When $\lambda < \frac{1}{(\sqrt{2\delta} + \sqrt{1-\delta})^2}$, we further design a novel polynomial-time greedy algorithm to achieve almost-exact recovery of the planted cycles. While enumerating all cycles in a graph is computationally intractable, our algorithm circumvents this bottleneck by leveraging structural insights from the generating function analysis and the decomposition of the symmetric difference $H^* \Delta H$.

A central observation is that 
when the generating function is bounded, any subgraph consisting of a disjoint union of cycles and a small number (i.e., $o(n)$) of paths, with a total of $\delta n - o(n)$ edges, differs from the planted cycles $H^*$ by at most $o(n)$ edges. Guided by this observation, our algorithm iteratively grows a subgraph that consists of a disjoint union of cycles and a small number of paths. It begins with an empty graph $H$ and incrementally grows it by searching through the collection of trails of length less than $\log n$ to identify ``good'' candidates $P$ to update $H$ to $H\Delta P$ (i.e., the XOR operation). These ``good'' trails, when XOR'ed onto $H$, either strictly increase $|H|$ without introducing new degree-$1$ vertices or increase $|H|$ by at least $\sqrt{\log n}$ while introducing at most $2$ new degree-$1$ vertices. 
Such ``good'' trails are guaranteed to exist thanks to the trail decomposition of $H^* \Delta H$: One can always take $P$ as either a full trail with more planted edges than unplanted ones or a segment of a trail with $\sqrt{\log n}$ more planted edges than unplanted ones. 
Importantly, these ``good'' trails can be efficiently identified through exhaustive search, as the total number of trails of length less than $\log n$ is $n^{O(1)}$ with high probability. 

It is instructive to contrast this setting with the planted clique problem, where similar greedy algorithms are expected to fail.  In that context, attempting to grow a clique $C$ from an empty graph by iteratively adding vertices connected to all current members typically stalls at size $\Theta(\log n)$: at that point, many vertices in $C$ are likely not part of the planted clique, and the chance of finding a vertex connected to all of $C$ becomes negligible, due to the independence of edge connections outside the planted structure. The algorithm thus becomes trapped in a local optimum.
By contrast, the planted cycles problem offers a more benign combinatorial landscape. Even if the current subgraph $H$ contains many unplanted edges, the trail decomposition of $H \Delta H^*$ ensures that ``good'' trails always exist to continue growing $H$, until getting close to $H^*.$

\subsection{Organization} 
The rest of this paper is structured as follows. Section \ref{sec:achievability} includes the achievability proofs, with Theorem \ref{thm:possibility_almost_exact_partial_2_factor} corresponding to the achievability statement of Theorem \ref{thm:almost-exact}. Section \ref{sec:impossibility} treats the impossibility results, with Theorem \ref{thm:impossibility-delta} corresponding to the impossibility statement of Theorem \ref{thm:almost-exact}. Finally, Section \ref{sec:algorithm} proposes a computationally efficient estimator that achieves the almost exact recovery threshold, as stated in~\prettyref{thm:alg}.
 Section \ref{sec:discussion} includes several directions for future work.

\section{Generating function bounds and proof of possibility of recovery}\label{sec:achievability}


The following lemma characterizes the difference graph of two $2$-factors, both on (possibly different) subsets of $m$ vertices. The proof follows from a general result in~\cite[Theorem 1]{kotzig1968moves}, which characterizes the existence of a Eulerian circuit that always traverses from one edge class to another edge class at every vertex.
\begin{lemma}\label{lemma:difference-graph}
Let $H^{\ast}$ be a planted $2$-factor on $m$ vertices, and let $H$ be another $2$-factor on a possibly different set of $m$ vertices. Then
\begin{itemize}
    \item The difference graph $H^{\ast} \Delta H$ is a graph where every vertex has degree $2$ or $4$. Furthermore, every degree-$4$ vertex has two incident red (planted) edges and two incident blue (unplanted) edges.
    \item 
    Furthermore, every connected component of  $H^{\ast} \Delta H$ has a Eulerian circuit whose edge color must alternate between red (planted) and blue (unplanted) at vertices in $V(H^*)\cap V(H).$  
\end{itemize}
\end{lemma}

The second part ensures that for any vertex in a circuit that is in the planted $2$-factor, at least one of the adjacent edges in the circuit is a planted edge.
\begin{proof}[Proof of Lemma \ref{lemma:difference-graph}]
To prove the first part, observe that for any $v \in [n]$, the sum of the degrees of $v$ in $H^{\ast}$ and $H$ is either $0$, $2$, or $4$. The degree of $v$ in $H^{\ast} \Delta H$ is equal to the sum of degrees of $v$ in $H^{\ast}$ and $H$, minus twice the number of edges incident to $v$ in $H^{\ast} \cap H$. It follows that $v$ has a degree of either $0$, $2$, or $4$ in $H^{\ast} \Delta H$. Observe that any degree-$4$ vertex in $H^{\ast} \Delta H$ must have no edges incident to $v$ in $H^{\ast} \cap H$, which implies that it has two red (planted) incident edges with the remaining two edges being blue (unplanted).

To prove the second part, we appeal to a general result of Kotzig \cite{kotzig1968moves}. Theorem 1 therein considers the setting of a connected graph $G = (V,E)$ where every vertex has an even degree and is associated with a labeled set of its incident edges. That is, a vertex $v$ is associated with a labeled partition of its incident edges $Q_v = (Q_{v,1}, Q_{v,2}, \dots)$, where $(v,w) \in Q_{v,i}$ means that the edge $(v,w)$ is given label $i$ relative to $v$. The result states that if $|Q_{v,i}| \leq \frac{1}{2} \deg(v)$ for all $v \in V$, then there exists a Eulerian circuit of $G$ such that for every pair of adjacent edges of the form $(u,v), (v,w)$, we have that $(u,v)$ and $(v,w)$ have different labels relative to $v$.

In our setting, we apply \cite[Theorem 1]{kotzig1968moves} to each connected component of $H^{\ast} \Delta H$. There are four different types of vertices, depending on the colors of the incident edges: 2 red, 2 blue, 1 red and 1 blue, or 2 red and 2 blue. The vertices belong to $V(H^*)\cap V(H)$ in the latter two cases but not in the first two.  In the first three cases, we simply classify the two incident edges into two label classes. For the third case, we put the red edges in one class and the blue edges in another class. The result then follows immediately from \cite[Theorem 1]{kotzig1968moves}.
\end{proof}

 \begin{figure}[ht]
   \begin{center}
\begin{tikzpicture}
    \def\pentagonSize{1.5cm}
    \draw[thick,red] (0:\pentagonSize) 
        \foreach \i in {72,144,216} {
            -- (\i:\pentagonSize) 
        }
        -- cycle; 
    \draw[thick,blue,dashed](0:\pentagonSize)
        \foreach \i in {144,216,288} {
            -- (\i:\pentagonSize)
        }
        -- cycle; 
     \node at (0,-2) {$H^* \cup H$};
\end{tikzpicture}
~~~~~~
\begin{tikzpicture}
    \def\pentagonSize{1.5cm}
    \draw[thick,red] (0:\pentagonSize) 
        \foreach \i in {72,144} {
            -- (\i:\pentagonSize) 
        };
      \draw[thick,red] (216:\pentagonSize) --(0:\pentagonSize); 
      \draw[thick,blue,dashed] (0:\pentagonSize) --(144:\pentagonSize); 
    \draw[thick,blue,dashed] (216:\pentagonSize)
        \foreach \i in {288,0} {
            -- (\i:\pentagonSize)
        }; 
    \node at (0,-2) {$H^* \Delta H$};
\end{tikzpicture}
\end{center}
\caption{An example where $H^*$ is a $4$-cycle with red solid edges, $H$ is a $4$-cycle with blue dashed edges, and $H^*\Delta H$ has a $(3,3)$-Eulerian circuit.}
\label{fig:illustration}
\end{figure}
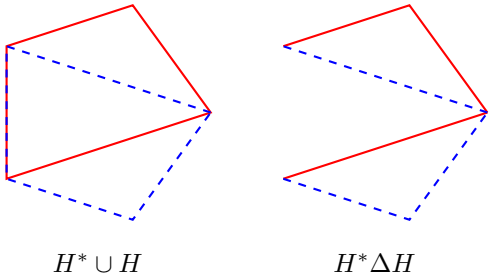

In light of Lemma~\ref{lemma:difference-graph}, $H^*\Delta H$ can be represented as a vertex-disjoint union of a certain class of circuits.  To specify that class of circuits, we introduce the definition of $(a,b)$-circuit, or more generally, $(a,b)$-trail. See~\prettyref{fig:illustration} for a graphical illustration.
\begin{definition}\label{def:trail}
Consider a graph $G$ with a planted $2$-factor  on a subset of its vertices. 
For $a\ge 0,b\ge 1$, 
an $(a,b)$-trail in $G$ is a trail in $G$ containing exactly $a$ planted edges and $b$ unplanted edges, such that the first edge is unplanted, the last edge is planted (if $a \geq 1$), no edge is repeated, and no two consecutive unplanted edges meet at a vertex that is contained in the planted $2$-factor. An $(a,b)$-path is an $(a,b)$-trail that does not repeat vertices. 
An $(a,b)$-circuit is a closed $(a,b)$-trail.
\end{definition}

The following lemma gives an upper bound on the expected number of $(a,b)$-trails, which will be used later to bound $\expect{|H^\ast \Delta H| }.$ 
\begin{lemma}\label{validWalkCountLem}
    
Let $G\sim \mathcal{G}(n,  \lambda, \delta)$ be the graph generated from the planted cycles model and let $H^*$ be the planted $2$-factor on the set of $\delta n$ vertices. 
Let $v$ and $v'$ be (not necessarily distinct) vertices. 
Fix $a \geq 1, b \ge 1$. Then, conditioned on $H^*$, the expected number of $(a,b)$-trails from $v$ to $v'$ in $G$ is at most $c_{a,b}/(\delta n)$,
where
\begin{align}
c_{a,b} \triangleq (\lambda(1-\delta))^b \sum_{k=1}^\infty (2\delta/(1-\delta))^k \binom{a-1}{k-1} \binom{b-1}{k-1} . \label{eq:def_c_a_b}
\end{align}
The expected number of $(0,b)$ trails from $v$ to $v'$ in $G$ is at most
$\left(1-\delta\right)^{b-1} \lambda ^b/n.$ 
\end{lemma}


\begin{proof}
We first handle the case where $a = 0$. We need to bound the expected number of trails containing exactly $b$ unplanted edges that start at $v$ and end at $v'$. Furthermore, the $b-1$ intermediate vertices must be unplanted. Since there are at most $((1-\delta)n)^{b-1}$ such potential trails, and each one is present with probability $\left(\frac{\lambda}{n}\right)^b$, the expected number of $(0,b)$ trails is at most 
$\left((1-\delta)n\right)^{b-1} \cdot \left(\frac{\lambda}{n}\right)^b = \left(1-\delta\right)^{b-1} \lambda ^b /n$.

Next, we consider $a \geq 1$. For any $k>0$ and $a_1,...,a_k,b_1,...,b_k>0$, let an $((a_1,...,a_k),(b_1,...,b_k))$-trail be an $(\sum_i a_i,\sum_i b_i)$-trail that starts with 
$b_1$ unplanted edges
followed by $a_1$ planted edges, then  $b_2$ unplanted edges, $a_2$ planted edges and so on. 
Let $\calW_{\vec{a}, \vec{b}}(v,v')$ denote the set of $(\vec{a},\vec{b})$-trails 
from $v$ to $v'$ in the complete graph $K_n$ (where edges in the complete graph are red if they are contained in $H^*$, and they are blue otherwise).
Then we let
$$
\calW_{a,b}(v,v')=\bigcup_{k \in \naturals_+} \bigcup_{\vec{a},\vec{b} \in \naturals_+^k:\|\vec{a}\|_1=a,\|\vec{b}\|_1=b} \calW_{\vec{a},\vec{b}}(v,v')
$$
denote the set of $(a,b)$-trails from $v$ to $v'$ in the complete graph $K_n.$
For each $(a,b)$-trail $T$ in $\calW_{a,b}(v,v')$, 
$
\prob{T \subset G} = (\lambda/n)^b.
$
Thus, 
$$
\expect{\sum_{T \in \calW_{a,b}(v,v')}
\indc{ T \subset G} }
= \left| \calW_{a,b}(v,v') \right| \left(\frac{\lambda}{n}\right)^b.
$$
We claim that for each $\vec{a}, \vec{b} \in \naturals_+^k,$
\begin{align}
\left| \calW_{\vec{a},\vec{b}}(v,v') \right| \le 2 (2\delta n)^{k-1} ((1-\delta)n)^{\|\vec{b}\|_1-k}. \label{eq:W_a_b_bound}
\end{align}
To see this, we first choose a set of directed paths of length $a_i$ in $H^*$ for each $1 \le i \le k$. For each of the first $k-1$ directed paths, there are at most $\delta n$ choices for the starting vertex, and two possible directions to traverse the planted cycles it is in. For the last directed path of length $a_k$ in $H^*$, since it must end at a given vertex $v'$ in $H^*$, the only choice is which direction to traverse the cycle that $v'$ is in. Thus there are at most $2^k(\delta n)^{k-1}$ choices in total for the set of planted directed paths. Now, given such a set, we are left to determine the set of directed paths of length $b_i$ consisting of unplanted edges for each $1 \le i \le k$. For each of these directed paths, since the starting and ending vertices have already been fixed, we only need to choose the unplanted vertices in the middle. There are in total $b-k$ unplanted vertices in the middles of paths, and thus there are at most $[(1-\delta)n]^{b-k}$ different choices.
In total, this results in at most $2^k (\delta n)^{k-1} [(1-\delta) n]^{b-k}$ distinct
$(\vec{a},\vec{b})$-trails 
from $v$ to $v'$ in the complete graph $K_n$.
Therefore, 
\begin{align}
\expect{\sum_{T \in \calW_{a,b}(v,v')}
\indc{ T \subset G} } 
&\le \sum_{k=1}^\infty \sum_{\vec{a},\vec{b} \in \naturals_+^k:\|\vec{a}\|_1=a,\|\vec{b}\|_1=b} 2^k (\delta n)^{k-1} ((1-\delta)n)^{b-k}\cdot (\lambda/n)^{b} \nonumber \\
&=\sum_{k=1}^\infty \sum_{\vec{a},\vec{b} \in \naturals_+^k:\|\vec{a}\|_1=a,\|\vec{b}\|_1=b} \frac{1}{\delta n} (2\delta/(1-\delta))^k (\lambda(1-\delta))^b \nonumber\\
& = \frac{1}{\delta n} (\lambda(1-\delta))^b \sum_{k=1}^\infty (2\delta/(1-\delta))^k \binom{a-1}{k-1} \binom{b-1}{k-1} = \frac{c_{a,b}}{\delta n},  \label{eq:c_a_b_step}
\end{align}
where the second equality holds because the number of distinct 
$\vec{a} \in \naturals_+^k$ such that $\|\vec{a}\|_1=a$ is $\binom{a-1}{k-1}.$
Therefore, the expected number of $(a,b)$-trails from $v$ to $v'$ in $G$ is at most $c_{a,b}/\delta n$ for any choice of $H^*$.
\end{proof}

To study $c_{a,b}$, we  consider its generating function; that is, 
\begin{align}
\sum_{a,b=1}^\infty c_{a,b}x^a y^b
& = \sum_{a,b=1}^\infty 
\sum_{k=1}^\infty \sum_{\vec{a},\vec{b} \in \naturals_+^k:\|\vec{a}\|_1=a,\|\vec{b}\|_1=b} (2\delta/(1-\delta))^k (\lambda(1-\delta))^b
x^a y^b  \nonumber \\
& =
\sum_{k=1}^\infty \sum_{\vec{a},\vec{b} \in \naturals_+^k} (2\delta/(1-\delta))^k (\lambda(1-\delta))^{\|\vec{b}\|_1}
x^{\|\vec{a}\|_1} y^{\|\vec{b}\|_1}  \nonumber \\
& = \sum_{k=1}^\infty
\prod_{i=1}^k
\left(\sum_{a_i,b_i=1}^\infty \frac{2\delta}{1-\delta} x^{a_i} ((1-\delta)\lambda y)^{b_i} \right) \nonumber \\
& = \sum_{k=1}^\infty\left(\frac{2x}{1-x}\cdot \frac{\delta\lambda y}{1-(1-\delta)\lambda y}\right)^k \triangleq g(x,y). \label{eq:generating_function}
\end{align}

If there exist $x$ and $y$ with 
$0<x<1$, $0<y<1/[\lambda(1-
\delta)]$, and $xy=1$ for which $\frac{2x}{1-x}\cdot \frac{\delta\lambda y}{1-(1-\delta)\lambda y}<1$ then $\sum_{k=1}^\infty\left(\frac{2x}{1-x}\cdot \frac{\delta\lambda y}{1-(1-\delta)\lambda y}\right)^k$ converges for these values of $x$ and $y$. That means that the expected number of $(a,b)$-circuits with $a=b$ is upper bounded by $g(x,y)$, and the expected number of $(a,b)$-circuits with $|a-b|=O(1)$ is at most $O(g(x,y)).$ With some more work, we can use that to prove that with probability $1-o(1)$, every $2$-factor on $\delta n$ vertices is nearly equal to the planted one. 

Conversely, if $\frac{2x}{1-x}\cdot \frac{\delta\lambda y}{1-(1-\delta)\lambda y}>1$ for all $x<1$, $y<1/[\lambda(1-
\delta)]$ with $xy=1$ then $\sum_{k=1}^\infty\left(\frac{2x}{1-x}\cdot \frac{\delta\lambda y}{1-(1-\delta)\lambda y}\right)^k$ diverges for all such $x$ and $y$. That can be used to show that the expected number of $(a,b)$-circuits with $a=b$ diverges as $n$ increases. This brings us to the question of when these cases hold, which has the following answer.
\begin{lemma}\label{lmm:lambdaThresholdLem}
If $\lambda<\frac{1}{(\sqrt{2\delta}+\sqrt{1-\delta})^2}$, then there exist 
$x$ and $y$ with $0<x<1$, $1<y<1/[\lambda(1-
\delta)]$, and $xy=1$ such that $\frac{2x}{1-x}\cdot \frac{\delta\lambda y}{1-(1-\delta)\lambda y}<1$. On the other hand, if $\lambda>\frac{1}{(\sqrt{2\delta}+\sqrt{1-\delta})^2}$, then
$\frac{2x}{1-x}\cdot \frac{\delta\lambda y}{1-(1-\delta)\lambda y}>1$ for all 
$x$ and $y$ with $0<x<1$, $1<y<1/[\lambda(1-
\delta)]$, and $xy=1$. 
\end{lemma}

\begin{proof}
If $xy=1$ then \begin{align*}
\frac{2x}{1-x}\cdot \frac{\delta\lambda y}{1-(1-\delta)\lambda y}&=\frac{2}{1-x}\cdot \frac{\delta\lambda}{1-(1-\delta)\lambda /x}=\frac{2\delta\lambda x}{(1-x)(x-(1-\delta)\lambda)}.
\end{align*}
At this point, the question of whether there is $(1-\delta) \lambda < x <1 $ for which this is less than $1$ is equivalent to the question of whether there is  $(1-\delta) \lambda < x <1 $ for which $(1-x)(x-(1-\delta)\lambda)>2\delta\lambda x$.
Let 
$$
f(x)=2\delta\lambda x - (1-x) \left(x-(1-\delta)\lambda \right)
= x^2+((3\delta-1) \lambda -1 )x+(1-\delta)\lambda.
$$
Note that $f(1)=2\delta \lambda \ge 0$, $f((1-\delta) \lambda))=2\delta \lambda^2(1-\delta) \ge 0 $, and 
$f(x)$ achieves its minimum
value of $f^*\triangleq (1-\delta)\lambda -((3\delta-1) \lambda -1 )^2/4$ at $x=x^*\triangleq (1- (3\delta-1)\lambda)/2$. Thus there exists $(1-\delta) \lambda < x <1 $ for which $f(x)<0$, provided that 
\begin{align}
    f^* <0 \text{ and } 
(1-\delta) \lambda \le x^* \le 1. \label{eq:condtion_lambda_check}
\end{align}

Note that $f^* <0$ is equivalent to $(3\delta-1)^2\lambda^2-(2\delta+2)\lambda+1>0$, which  
happens when $\lambda < \lambda_-$ or $\lambda > \lambda_+, $ where 
\begin{align*}
\lambda_{\pm} = \frac{(2\delta+2)\pm \sqrt{(2\delta+2)^2-4(3\delta-1)^2}}{2(3\delta-1)^2}&=\frac{\delta+1\pm \sqrt{(\delta+1)^2-(3\delta-1)^2}}{(3\delta-1)^2} =\frac{\delta+1 \pm \sqrt{8\delta-8\delta^2}}{(3\delta-1)^2}\\
&=\frac{(\sqrt{2\delta}\pm \sqrt{1-\delta})^2}{(3\delta-1)^2}=\frac{1}{(\sqrt{2\delta}\mp\sqrt{1-\delta})^2}.
\end{align*}
Also, $(1-\delta) \lambda \le x^* \le 1$ is equivalent to $(1+\delta) \lambda \le 1$. 
Note that $\lambda_-\le \frac{1}{1+\delta}$ and $\lambda_+\ge \frac{1}{1+\delta}$. Thus, \prettyref{eq:condtion_lambda_check} holds provided that $\lambda < \lambda_- = \frac{1}{(\sqrt{2\delta}+\sqrt{1-\delta})^2}.$

Conversely, if $\lambda>\lambda_-=\frac{1}{(\sqrt{2\delta}+\sqrt{1-\delta})^2}$, then either 
$f^*>0$
or $(1-\delta) \lambda \ge x^* $. Thus $f(x)>0$ for all 
$(1-\delta) \lambda < x <1 $.

\end{proof}


The previous lemma finally allows us to prove that we can achieve almost exact recovery whenever $\lambda<\frac{1}{(\sqrt{2\delta}+\sqrt{1-\delta})^2}$.

\begin{theorem}\label{thm:possibility_almost_exact_partial_2_factor}
Let $\delta\in (0,1]$ and $0<\lambda<\frac{1}{(\sqrt{2\delta}+\sqrt{1-\delta})^2}$. Suppose $G\sim \calG(n, \lambda,\delta)$ is the graph generated from the planted cycles model conditional on $H^*$.
There exists $C$ independent of $n$ such that $\mathbb{E}[ |H^*\Delta H|]\le C$
for any $2$-factor $H$ on $\delta n$ vertices in $G$.  
\end{theorem}

\begin{proof}
Recall that $H^*\Delta H$ can be represented as a vertex-disjoint union of $(a,b)$-circuits such that exactly half of the total edges are planted. However, each individual $(a,b)$-circuit may be imbalanced, \ie, have an unequal number of planted and unplanted edges. Nevertheless, we will argue that the expected number of $(a,b)$-circuits with more than $(1/2-\epsilon)$ fraction of their edges unplanted is bounded. 

Towards this end, for each $(a,b)$-circuit $C$ in graph $G$,
define its excess $\ex(C)=b-(1/2-\epsilon) (a+b)$, where $\epsilon>0$ will be specified later. Then define the total excess over all possible valid circuits with positive excess as 
$$
\Gamma = \sum_{C \subset G}
\ex(C) \indc{\ex(C) > 0},
$$
Given any $2$-factor on $\delta n$ vertices in $G$, suppose that 
$H^\ast \Delta H$ is the vertex-disjoint union of $(a_i,b_i)$-circuits $C_i$ for $1 \le i \le m$. Then by definition,
\begin{align*}
\frac{1}{2} |H^\ast \Delta H| = 
\sum_{i=1}^m b_i 
& = \sum_{i=1}^m 
\left[b_i - \left(1/2 - \epsilon\right)(a_i+b_i) \right] + 
(1/2-\epsilon) |H^\ast \Delta H| \\
& = \sum_{i=1}^m \ex(C_i) 
+ (1/2-\epsilon) |H^\ast \Delta H| \\
& \le \Gamma + (1/2-\epsilon) |H^\ast \Delta H|.
\end{align*}
It follows that $|H^\ast \Delta H| \le \Gamma /\epsilon$. 

It remains to bound $\expect{\Gamma }.$ 
In order to do that, we decompose $\Gamma=\Gamma_0+\Gamma_1$, where 
$\Gamma_0$ sums over the $(0,b)$-circuits $C$ and 
$\Gamma_1$ sums over the $(a,b)$-circuits $C$ for $a \ge 1.$

By Lemma~\ref{validWalkCountLem}, the expected number of $(0,b)$-trails from $v$ to $v'$ in $G$ is at most  $\lambda^b (1-\delta)^{b-1} /n$. Since a $(0,b)$-circuit is a special case of a $(0,b)$-trail that starts and ends at a vertex $v $ not in $H^*$, it follows that the expected number of 
$(0,b)$-circuits is at most $\lambda^b (1-\delta)^{b-1} /n \times (1-\delta) n = \lambda^b (1-\delta)^b.$
Therefore,
\begin{align}
\mathbb{E}[\Gamma_0 ]\le (1/2+\epsilon)  \sum_{b\ge 1} b(\lambda(1-\delta))^{b} = (1/2+\epsilon) \frac{\lambda(1-\delta)}{(1-\lambda(1-\delta))^2}
\end{align}
where the last equality holds because $\sum_{b \ge 1} b x^b=x/(1-x)^2$ for $|x|<1$
and $\lambda(1-\delta)<1$.

Next, we bound $\expect{\Gamma_1 }.$ By Lemma \ref{validWalkCountLem},
for $a, b \ge 1$, the expected number of $(a,b)$-trails from $v$ to $v'$ in $G$ is at most  $c_{a,b}/(\delta n)$. Since an $(a,b)$-circuit is a special case of an $(a,b)$-trail that starts and ends at a vertex $v$ in $H^*$, it follows that the expected number of 
$(a,b)$-circuits is at most $c_{a,b}/(\delta n) \times \delta n = c_{a,b}.$ It follows that 
\begin{align*}
\mathbb{E}[\Gamma_1 ] \le \sum_{a,b \ge 1 } [b- (1/2-\epsilon)(a+b)] c_{a,b} \cdot \indc{b> (1/2-\epsilon)(a+b)}.
\end{align*}
Now, we bound the RHS of the last display equation using the generating function of $c_{a,b}$. Recall that by \prettyref{lmm:lambdaThresholdLem}, there exist $x_0,y_0$ such that $0<x_0<1$, $0<y_0<\frac{1}{\lambda(1-\delta)}$, $x_0y_0=1$, and $\frac{2x_0}{1-x_0}\cdot\frac{\delta \lambda y_0}{1-(1-\delta)\lambda y_0}<1$. Next, observe that there must exist $y>y_0$ such that $\frac{2x_0}{1-x_0}\cdot\frac{\delta \lambda y}{1-(1-\delta)\lambda y}<1$. Also, let $x=x_0$. At this point, $x^2<1<xy$ so there must exist some $0<\epsilon<1/2$ such that $x^{1+2\epsilon}y^{1-2\epsilon}=1$. 
Therefore, for all $b \ge (1/2-\epsilon)(a+b),$
\begin{align*}
 b- (1/2-\epsilon)(a+b)
\le \frac{y/x}{y/x-1} 
\left(y/x\right)^{b- (1/2-\epsilon)(a+b)} 
= \frac{y/x}{y/x-1}   x^ay^b,
\end{align*}
where the inequality holds because $ t^{\alpha+1} \ge 1+(\alpha+1)(t-1) \ge \alpha (t-1)$ for $t \ge 1$ and $\alpha \ge 0$; the equality holds due to
$x^{1+2\epsilon}y^{1-2\epsilon}=1$. Combining the last displayed equations gives that 
\begin{align*}
\mathbb{E}[\Gamma_1 ]
&\le \frac{y/x}{y/x-1} \sum_{a,b\ge 1} 
c_{a,b} x^ay^b\\
&= \frac{y/x}{y/x-1} \sum_{k=1}^\infty\left(\frac{2x}{1-x}\cdot \frac{\delta\lambda y}{1-(1-\delta)\lambda y}\right)^k \\
& \le \frac{y/x}{y/x-1}\frac{1}{1-\frac{2x}{1-x}\cdot \frac{\delta\lambda y}{1-(1-\delta)\lambda y}},
\end{align*}
 
In conclusion, 
$\expect{|H^\ast \Delta H| } \le 
\expect{\Gamma } /\epsilon \le C$ for some constant $C$ only depending on $\epsilon,\lambda, \delta$. 

\end{proof}

\section{Generating functions to prove impossibility of recovery}\label{sec:impossibility}
In this section, we will prove the following theorem, establishing that when $\lambda>\frac{1}{(\sqrt{2\delta}+\sqrt{1-\delta})^2}$, any estimator must suffer at least constant error.
\begin{theorem}\label{thm:impossibility-delta}
If 
\begin{align}
    \lambda  >\frac{1}{(\sqrt{2\delta}+\sqrt{1-\delta})^2}, \label{eq:impossibility-delta}
\end{align}
then there exists a constant $\epsilon> 0$ depending only on $\lambda$ and $\delta$ such that
 for any estimator $\widehat{H}$, 
 \begin{align}\mathbb{P}\left(\ell(H^*, \hat{H}) \geq \epsilon/2 \right) \ge 1 -n^{-\Omega(\log n)}.
 \label{eq:error_lower_bound_prob}
 \end{align}
It follows that for all sufficiently large $n$, 
$$
\expect{\risk(\hat{H}, H^*)}\ge \epsilon/4.
$$
\end{theorem}

The proof of Theorem \ref{thm:impossibility-delta} is similar in spirit to the proof of \cite[Theorem 2.3]{gaudio2025all}. Let $\mu_G$ denote the posterior distribution:
\begin{align}
\mu_G(H) \triangleq \prob{H^*=H \mid G } = \frac{1}{|\calH(G)|} \indc{H \in \calH(G)}, \quad \forall H \in \calH(G),  \label{eq:posterior_distribution}
\end{align}
where $\calH(G)$ denotes the set of degree-$2$ subgraphs on $\delta n$ vertices. That is, $\mu_G(\cdot)$ is the uniform distribution over degree-$2$ subgraphs on $\delta n$ vertices. 
A  simple yet crucial observation is that while a random draw $\tilde{H}$ from the posterior distribution~\eqref{eq:posterior_distribution} may not minimize the reconstruction error, its reconstruction error can be related to that of any estimator $\hat{H}$ as follows: for any $D>0$,
\begin{align}
\mathbb{P}\left(\ell(H^*,\hat{H})< D \right) \le 
\sqrt{\mathbb{P}\left(\ell(H^*,\tilde{H}) < 2D \right) }. \label{eq:posterior-tail}
\end{align}
(see~\cite[Eq. 3.2]{gaudio2025all} for a short proof).
Therefore, it suffices to prove~\prettyref{eq:error_lower_bound_prob} holds for the posterior sample $\tilde{H}$, which further reduces to demonstrating that the observed graph $G$ contains many more $H \in \mathcal{H}$ that are far from $H^*$ than those close to $H^*$.

We begin by defining two sets, one containing subgraphs $H$ that are similar to $H^*$, and the other containing subgraphs $H$ that are markedly different from $H^*$.
Concretely, given $\epsilon>0$, define
\begin{align*}
\Mgood  = & ~ \left\{ H \in \calH: \ell( H, H^*) < \epsilon, H \subset G \right\} \\
\Mbad = & ~ \left\{ H \in \calH: \ell(H,H^*)  \geq \epsilon, H \subset G \right\}.
\end{align*}
Using the fact that the posterior distribution $\mu_G$ is the uniform distribution over all possible $k$-factors contained in the observed graph $G$, we have 
\[
\prob{\ell(\tilde H,H^*) < \epsilon \mid G, H^*} = \frac{|\Mgood|}{ |\Mgood| + |\Mbad| }.
\]
The following lemmas give high probability bounds on $|\Mgood|$ and $|\Mbad|$.
\begin{lemma}
\label{lmm:Mgood-delta}	
For $\epsilon< \min\{1, 2\lambda/\delta\},$ conditioned on $H^*$, with probability at least $1-e^{-\Omega(n)}$, 
\begin{equation*}
\left| \Mgood \right|\leq \frac{\epsilon \delta n}{2} \left(\frac{2e}{\epsilon} \cdot \frac{4e \lambda }{\epsilon \delta} \right)^{\epsilon \delta n/2}.
\end{equation*}

 
\end{lemma}
\begin{proof}
With probability $1-e^{-\Omega(n)}$, there are at most $2\lambda n$ unplanted edges. Conditioned on this event, we bound the cardinality of $\Mgood$. Observe that for $H \in \Mgood$, $|H|=\delta n$
and $|H\Delta H^*|\le \epsilon |H^*|= \epsilon \delta n$. To determine $H$, it suffices to choose the planted edges for $H^*\setminus H$
and unplanted edges for $H\setminus H^*.$
Therefore, there are at most $\binom{\delta n}{i} \binom{2\lambda n}{i}$ such graphs $H$ with set difference exactly $2i$ relative to $H^{\ast}$. Hence, we can bound $|\Mgood|$ by
\begin{align*}
\sum_{i=1}^{\epsilon \delta n/2} \binom{\delta n}{i} \binom{2\lambda n}{i}
    &\le \frac{\epsilon \delta n}{2}
    \binom{\delta n}{\epsilon \delta n/2} \binom{2\lambda n}{\epsilon \delta n/2} \\
    &\leq \frac{\epsilon \delta n}{2} \left(\frac{2e}{\epsilon} \cdot \frac{4e\lambda }{\epsilon \delta} \right)^{\epsilon \delta n/2}  .  
\end{align*}
Here the first inequality holds since by assumption $\epsilon< \min\{1, 2\lambda/\delta\}$, which implies that the maximal summand is $\binom{\delta n}{\epsilon \delta n/2} \binom{2 \lambda n}{\epsilon \delta n/2}$.
The second step uses $\binom{n}{k} \le (en/k)^k.$
\end{proof}

\begin{lemma}
\label{lmm:Mbad-delta}
Assume that \eqref{eq:impossibility-delta} holds. Then there exist constants $c_1$ and $c_2$ that only depend on $\lambda, \delta$, 
such that for all $\epsilon \leq c_2/\delta$,  conditioned on $H^*$ containing at most $(\log_2 n)^2$ cycles, 
with probability at least $1-e^{-\Omega(n)}$,
		\begin{equation}
	\left| \Mbad  \right| \geq e^{c_1 n}.
	\label{eq:Mbad}
	\end{equation}	
\end{lemma}

\begin{proof}[Proof of~\prettyref{thm:impossibility-delta}]
Combining Lemmas
\ref{lmm:Mgood-delta} and \ref{lmm:Mbad-delta}, we obtain that conditioned on $H^*$ containing at most $(\log_2 n)^2$ cycles, with probability $1-e^{-\Omega(n)}$,
 $$
 \frac{|\Mgood |}{ |\Mgood| + |\Mbad| } \leq  \frac{|\Mgood |}{ |\Mbad| } \leq \frac{\epsilon \delta n}{2} \left(\frac{2e}{\epsilon} \cdot \frac{4e \lambda }{\epsilon \delta} \right)^{\epsilon\delta n/2} \cdot e^{-c_1 n}.
 $$ 
 Setting $\epsilon > 0 $ to be sufficiently small, we can bound the right-hand side by $n \cdot e^{-c_1 n/2}$.
 Therefore, 
  \begin{align*}
    \left[\prob{\ell(H^*, \hat{H}) < \epsilon/2} \right]^2 &\leq \prob{\ell(H^*, \tilde{H}) < \epsilon }\\
    &= \mathbb{E}_{G, H^*} \left[ \prob{\ell(H^*, \tilde{H}) < \epsilon \mid G, H^* }\right]\\
    &= \mathbb{E}_{G, H^*} \left[\frac{|\Mgood |}{ |\Mgood| + |\Mbad| }   \right]\\
    &\leq \left(1 - e^{-\Omega(n)}\right) n  e^{-c_1n/2} +
    e^{-\Omega(n)} + \prob{\calE^c} = n^{- \Omega(\log n)},
 \end{align*}
 where the first inequality comes from~\prettyref{eq:posterior-tail}; $\calE$ denotes the event that $H^*$ contains at most $(\log_2 n)^2$ cycles;
 and the last inequality holds because $\prob{\calE^c} \le n^{-\Omega(\log n)}$ in view of~\prettyref{lmm:factorCycleLem}.
\end{proof}

The remainder of this section is devoted to proving Lemma \ref{lmm:Mbad-delta}. Building on the circuit decomposition in Lemma \ref{lemma:difference-graph}, our goal is to show that with high probability, there exist exponentially many $\Theta(n)$-length cycles containing an equal number of planted and unplanted edges—referred to as balanced cycles. Perhaps unexpectedly, these long balanced cycles can be assembled from constant-length paths that are themselves balanced, containing the same number of planted and unplanted edges.


We will search for these constant-length, balanced paths by constructing a linear number of trees, each containing many constant-length paths. 
Specifically, in Lemma \ref{lmm:m-star}, we first show that there exists a constant $m^{\ast}$ for which $c_{m^{\ast},m^{\ast}}>1$. 
Moreover, in~\prettyref{lmm:lower_bound_path}, we show that for any $a, b \ge 1$, the expected number of $(a,b)$-paths starting at a given vertex converges to
$c_{a,b}.$ Therefore, we can define a branching process, where each step goes from a planted vertex to all other planted vertices reachable from it by an $(m^{\ast},m^{\ast})$-path. The expected number of children of each node in this process will be greater than $1$, enabling us to show that its probability of dying out is bounded away from $0$. That allows us to construct many trees where each path to a leaf has an equal number of planted and unplanted edges.

Then, we apply a sprinkling argument, connecting the trees to form an exponentially large number of cycles that have equal numbers of planted and unplanted edges. Similar sprinkling constructions appeared in \cite{Ding2023,gaudio2025all}, with \cite{gaudio2025all} being the most similar to the present argument.

In \cite{gaudio2025all}, which studies the case of $\delta=1$, the impossibility of almost exact recovery coincides with the presence of many long \emph{alternating} cycles. Accordingly, their tree construction procedure searches for many trees that alternate between planted and unplanted edges across layers. In contrast,  when $\delta<1$, impossibility is driven instead by the presence of many long  \emph{balanced} cycles, which need not be alternating. Fortunately, we find that there still exist sufficiently many balanced cycles containing segments of length $2m^*$ comprising $m^*$ planted edges and $m^*$ unplanted edges, where $m^*$ is some constant depending on $\lambda$ and $\delta.$  If $\delta=1,$ one can simply take $m^* = 1$, recovering the case of alternating cycles. When $\delta < 1$, however, we need to take $m^* > 1$. Our tree construction procedure therefore attempts to find many trees built from such balanced segments (referred to as \emph{path layers}) of length $2m^{*}$. The challenge is that adding a path layer at each step requires exploring neighborhoods of radius $2m^*$, which introduces dependencies and significantly complicates the analysis, as we will see.

In what follows, we describe our tree construction procedure in~\prettyref{sec:tree_construction} and analyze it in~\prettyref{sec:tree_construction_analysis}. The sprinkling argument is presented  in~\prettyref{sec:sprinkling}.

\subsection{Tree construction} \label{sec:tree_construction}
Recall that when $\lambda>\frac{1}{(\sqrt{2\delta}+\sqrt{1-\delta})^2}$, we have that $\frac{2x}{1-x}\cdot\frac{\delta\lambda y}{1-(1-\delta)\lambda y}>1$ for all $x\in(0,1)$ and $y\in (0,1/(1-\delta)\lambda)$ such that $xy=1$ due to Lemma \ref{lmm:lambdaThresholdLem}. The following lemma further establishes that the power series $\sum_{k=1}^\infty\left(\frac{2x}{1-x}\cdot \frac{\delta\lambda y}{1-(1-\delta)\lambda y}\right)^k$ has a coefficient $c_{m^*,m^*}$ satisfying $c_{m^*,m^*} > 1$.

\begin{lemma}\label{lmm:m-star}
Let $0<\delta\le 1$ and $\lambda>\frac{1}{(\sqrt{2\delta}+\sqrt{1-\delta})^2}$. Then there exists $m^\ast$ such that 
$c_{m^*,m^*}>1.$ In particular, when $\delta=1$,  and $\lambda>1/2$, we have $c_{a,a}=(2\lambda)^a$ and  $m^*$ can be chosen as $1.$
\end{lemma}

\begin{proof}

For the special case of $\delta=1$, by the definition  in~\prettyref{eq:def_c_a_b},
$
c_{a,a}= (2\lambda)^a
$.
Thus when $\lambda>1/2$, $m^*$ can be chosen as $1.$ In general, note that $c_{a,b}$ is increasing in $\lambda$, so without loss of generality, we assume $\frac{1}{(\sqrt{2\delta}+\sqrt{1-\delta})^2}< \lambda<\frac{1}{1-\delta}$.
Let 
$$
f(x,y)=\frac{2x}{1-x}\cdot\frac{\delta\lambda y}{1-(1-\delta)\lambda y}
$$
and $h(x)=f(x,1/x).$
Let $x^*$ be a value of $x$ in $((1-\delta)\lambda,1)$ that minimizes $h(x).$ 
By~\prettyref{lmm:lambdaThresholdLem}, $h(x^*)>1$. 
Next, define the finite series expansion of $f(x,y)$ as
\[
f_m(x,y) \triangleq 2\sum_{i=1}^m x^i \cdot \sum_{j=1}^m \left((1-\delta)\lambda y \right)^{j-1} \delta\lambda y
=\frac{2\delta}{1-\delta} \sum_{i=1}^m x^i \cdot \sum_{j=1}^m \left((1-\delta)\lambda y \right)^{j} \]
and $h_m(x)=f_m(x,1/x).$
Note that as $m \to \infty,$ $h_m(x)$
 converges uniformly to 
 $h(x)$
on any closed interval contained in $((1-\delta)\lambda,1)$. So, there exists $m_0$ such that $h_{m_0}(x)$ has a local minimum strictly greater than $1$ at some $x_0 \in ((1-\delta)\lambda, 1)$; that is, $h'_{m_0}(x_0)=0$ and $h_{m_0}(x_0)>1$. Note that 
\begin{align*}
h'_m(x) & = \frac{\partial}{ \partial x} f_m(x,y) \vert_{y=1/x}
+ \frac{\partial }{\partial y} f_m(x,y) \vert_{y=1/x} \cdot \frac{d(1/x)}{dx} \\
& = \frac{1}{x} \left( x \frac{\partial}{\partial x} - y \frac{\partial }{ \partial y} \right) f_m(x,y)\vert_{y=1/x} \\
& = \frac{2\delta}{x(1-\delta)} \sum_{i,j=1}^m (i-j) x^i \left((1-\delta)\lambda y \right)^{j} 
\vert_{y=1/x}. 
\end{align*}
In particular, 
$h'_{m_0}(x_0)=0$ implies that
\begin{align}
\sum_{i,j=1}^{m_0} (i-j) x^i \left((1-\delta)\lambda y \right)^{j}
\vert_{x=x_0,y=1/x_0} = 0. 
\label{eq:mean_same}
\end{align}
Now, consider picking a random $(\mathbf{I},\mathbf{J})$ with probability given by
\[\mathbb{P}[(\mathbf{I},\mathbf{J})=(i,j)]
= \frac{1}{h_{m_0}(x_0) } \frac{2\delta}{1-\delta} x^{i}((1-\delta)\lambda y)^{j}\vert_{x=x_0,y=1/x_0}  \]
for all $1\le i,j \le m_0$.
Then it follows from~\prettyref{eq:mean_same} that $\expect{\mathbf{I}}=\expect{\mathbf{J}}$.
That means that the exponents of $x$ and $y$ in $f_{m_0}(x,y)$ are in some sense balanced, but ultimately we want to prove the existence of a large coefficient on a single term with equal powers of $x$ and $y$. To this end, we consider a sequence of i.i.d.\ copies of $(\mathbf{I},\mathbf{J})$, that is, $(\mathbf{a}_1,\mathbf{b}_1), \ldots, (\mathbf{a}_\ell, \mathbf{b}_\ell) \iiddistr (\mathbf{I},\mathbf{J})$. 
Let $\mathbf{a}=\sum_{i=1}^\ell
\mathbf{a}_i$ and 
$\mathbf{b}=\sum_{i=1}^\ell
\mathbf{b}_i$.
By Chebyshev's inequality, 
with probability at least $1/2,$ 
$$
\left| \mathbf{a}- \ell  \mathbb{E}[\mathbf{I}] \right| < c\sqrt{\ell } \quad \text{ and } \quad \left| \mathbf{b}- \ell \mathbb{E}[\mathbf{J}] \right| < c\sqrt{\ell}, \quad \text{ where } c=2\sqrt{\max \left(\var[\mathbf{a}_1],\var[\mathbf{b}_1] \right)}.
$$
Let $m^*= \lceil \ell  \mathbb{E}[\mathbf{I}]+ c\sqrt{\ell}\rceil$.
Recall from~\prettyref{eq:c_a_b_step} that 
\begin{align*}
c_{m^*, m^*}&=
(\lambda(1-\delta))^{m^*} \sum_{k=1}^\infty
 (2\delta/(1-\delta))^k  \sum_{\vec{a},\vec{b} \in \naturals_+^k:\|\vec{a}\|_1=m^*,\|\vec{b}\|_1=m^*} 1 \\
& \ge  (\lambda(1-\delta))^{m^*}  (2\delta/(1-\delta))^{\ell+1}
\sum_{a, b=1}^{m^*-1} \sum_{\vec{a},\vec{b} \in \naturals_+^\ell:\|\vec{a}\|_1=a,\|\vec{b}\|_1=b} 1 \\
& = (2\delta/(1-\delta))^{\ell+1} (\lambda(1-\delta))^{m^*} 
 \sum_{a, b=1}^{m^*-1} 
 \sum_{\vec{a},\vec{b} \in \naturals_+^\ell:\|\vec{a}\|_1=a,\|\vec{b}\|_1=b}  x^{m^*} y^{m^*} \vert_{x=x_0, y=1/x_0} \\
 & \ge 
 (2\delta/(1-\delta))^{\ell+1} (x\lambda(1-\delta)y)^{\lceil 2c\sqrt{\ell}\rceil } 
 \sum_{a, b = m^*-\lceil 2c\sqrt{\ell}\rceil}^{m^*-1} 
 \sum_{\vec{a},\vec{b} \in \naturals_+^\ell:\|\vec{a}\|_1=a,\|\vec{b}\|_1=b}  x^{a}  (\lambda(1-\delta)y)^b \vert_{x=x_0, y=1/x_0}\\
 & \ge \frac{2\delta}{1-\delta} \left(h_{m_0}(x_0) \right)^\ell (\lambda(1-\delta))^{\lceil 2c\sqrt{\ell}\rceil } 
 \cdot \prob{\left| \mathbf{a}- \ell  \mathbb{E}[\mathbf{I}] \right| < c\sqrt{\ell }, \left| \mathbf{b}- \ell \mathbb{E}[\mathbf{J}] \right| \le c\sqrt{\ell}}\\
 &\ge \frac{2\delta}{1-\delta} \left(h_{m_0}(x_0) \right)^\ell (\lambda(1-\delta))^{\lceil 2c\sqrt{\ell} \rceil} \cdot \frac{1}{2} ,
\end{align*}
where the first inequality holds by fixing $k=\ell+1$ and choosing $a_{\ell+1}=m^*-a$
and $b_{\ell+1}=m^*-b$;
5he second inequality holds by restricting the sum over $a,b$
to the range $a,b \ge m^*-2c\sqrt{\ell}$, and using the facts that $x_0<1,
\lambda(1-\delta)/x_0<1.$
Finally, since $h_{m_0}(x_0)>1$, by choosing $\ell$ sufficiently large, we have
$c_{m^*,m^*}>1.$

\end{proof}


Later, in~\prettyref{lmm:lower_bound_path}, we show that for any $a, b \ge 1$, the expected number of $(a,b)$-paths starting at a given planted vertex is asymptotically $c_{a,b}$. Combining this with~\prettyref{lmm:m-star}, we conclude that when $\lambda > \frac{1}{(\sqrt{2\delta}+\sqrt{1-\delta})^2}$, the expected number of $(m^\ast,m^\ast)$-paths from any given planted vertex exceeds $1$.
This observation allows us to construct a suitable collection of trees, referred to as \emph{path trees}, as illustrated in Figure \ref{fig:path-trees}.  In particular, when $\delta=1$,  $m^*$ can be chosen as $1$, and each purple path in the path tree consists of one unplanted edge followed by one planted edge, thereby recovering the tree construction in~\cite{gaudio2025all}.  
\begin{figure}[h]
    \centering
    \includegraphics[scale=0.3]{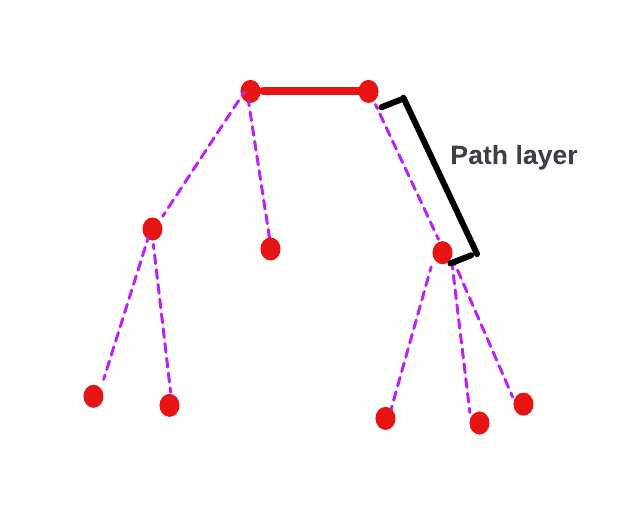}
    \caption{A schematic representation of a path tree with a depth of $2$. The red nodes are called ``hub'' nodes, and the dashed purple lines represent $(m^{\ast}, m^{\ast})$-paths, where $m^{\ast}$ is given in Lemma \ref{lmm:m-star}. A purple path along with the hub node below it is referred to as a ``path layer'' of the path tree.}
    \label{fig:path-trees}
\end{figure}

The construction of path trees is given in Algorithm \ref{alg:tree-construction2}.
The algorithm operates on a pre-selected set of \emph{available} vertices, owing to a pre-processing step that reserves edges for later use in linking trees together (Algorithm \ref{alg:matching}). The available vertices are precisely those that are not endpoints of reserved edges.
\begin{algorithm}[h]
\caption{Tree Construction} \label{alg:tree-construction2}
\begin{algorithmic}[1]
\Statex{\bfseries Input:} Graph $G_0 \cup H^{\ast}$ on $n$ vertices, available vertices $\mathcal{A}$ with $|\mathcal{A}|=(1-2\gamma) n$ for some $\gamma>0$, length $m^\ast$,  size parameter $\ell \in \mathbb{N}$
\Statex{\bfseries Output:} A set of two-sided trees where each side has at least $2\ell$ hub nodes. 

\smallskip
\State Set $\mathcal{T} = \emptyset$.
\For{$t \in \left\{1, 2, \dots, K := \frac{\gamma n}{\ell}  \right\}$} \label{line:K} 
\State \label{line:fail} Select a random planted edge $(u_0,u_0')$ where $u_0,u_0'\in \mathcal{A}$. If no such edge exists, return $\texttt{FAIL}$. 
\State Initialize $T$ to be a two-sided tree containing only the center edge $(u_0,u_0')$.
\State Remove $u_0$ and $u_0'$ from $\mathcal{A}$.
\State \label{line:right-1} (Grow the left tree rooted at $u_0$.) Initialize the leaf queue to be $\mathcal{L} \gets \{u_0\}$, and the cumulative size to be $s \gets 1$.
\While{$\mathcal{L} \neq \emptyset$ and $s < 2\ell$}
\State Let $u \gets \mathcal{L}.\texttt{pop}$. 
\State (Find the descendants of $u$.)\label{line:explore} Let $\mathcal{C}_u$ be the set of all vertices $v$ such that there is a 
non-shortcutted $(m^\ast,m^\ast)$-path from $u$ to $v$ with all of its vertices except $u$ in $\mathcal{A}$.

\State Add the path layer (the unique $(m^*,m^*)$-path) from $u$ to every vertex in $\mathcal{C}_u$ into the current tree $T$.

\State (Prune available vertices) \label{line:remove1} 
Let $\mathcal{A}_0$ be a copy of $\mathcal{A}$. Prune the set $\mathcal{A}$ by removing every vertex that is reachable from $u$ by a path of length at most $2m^{\ast}$ contained in $\mathcal{A}_0$.
\State Set $s \gets s + |\mathcal{C}_u|$, and update $\mathcal{L}$ as $\mathcal{L}.\texttt{push}(\mathcal{C}_u)$

\EndWhile \label{line:right-2}
\If{$s \geq 2\ell$}
\State\label{line:remove2}Grow the right tree rooted at $u_0'$ by following Lines \ref{line:right-1} to \ref{line:right-2} with $u_0$ replaced by $u_0'$.
\State If the right tree also reaches a size of at least $2\ell$, then let $T$ be the resulting two-sided tree, and set $\mathcal{T} \gets \mathcal{T} \cup \{T\}$. 
\EndIf
\EndFor
\State Return $\mathcal{T}$.
\end{algorithmic}
\end{algorithm}
\FloatBarrier

During the process of constructing trees, one must take care to ensure that any blue edges remain independent $\text{Bern}(\lambda/n)$ random variables. Line \ref{line:remove1} ensures that whenever the neighborhood of a hub $u$ is explored, that neighborhood will not be accessible to exploration from future hubs, and thus ensures the required independence of blue edges in future explorations. A difficulty arises when $u, v$ are hub nodes as in Figure \ref{fig:tree-problem}. In the situation described therein, the blue neighbors of $v$ would be removed by Line \ref{line:remove1}, and thus the required independence for exploration from $v$ would be lost. A similar issue arises when there are at least two paths of length $2m^{\ast}$ from $u$ to $v$ (even if only one is an $(m^{\ast}, m^{\ast}$-path). Therefore,  to avoid this issue, in  Line \ref{line:explore}, we only consider \emph{``non-shortcutted'' paths}, meaning that any other path connecting the same endpoints must be strictly longer. In other words, a path is ``shortcutted'' if there exists another path between its two endpoints that is of the same or shorter length.

Some restraint is required when building the trees; it turns out that we cannot simply take the number of hub nodes $\ell$ to be linear, in the hopes of building cycles from a few trees.
The value of $\ell$ is limited due to the pruning step in Line \ref{line:remove1}. In particular, taking $\ell$ to be linear (or even super-constant) would remove too many vertices for every tree, leading to few trees being constructed. In other words, it is better to create many small trees rather than a few large trees.  

\begin{figure}
    \centering
    \includegraphics[scale=0.3]{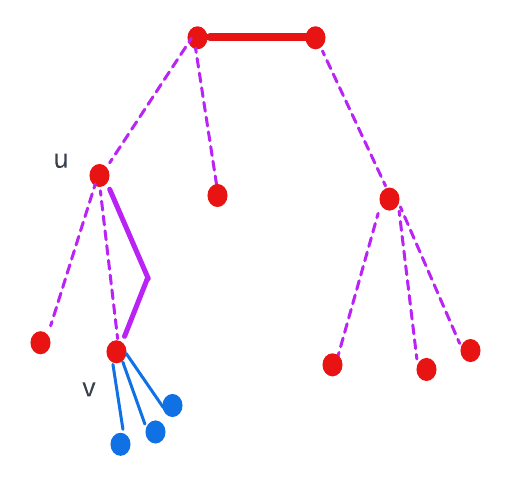}
    \caption{The solid purple line represents a path of length strictly less than $2m^{\ast}$.}
    \label{fig:tree-problem}
\end{figure}

\subsection{Analysis of tree construction}
\label{sec:tree_construction_analysis} 
In this subsection, we prove that~\prettyref{alg:tree-construction2} successfully builds $\Theta(n)$ many two-sided trees, conditioned on $H^*$ not containing too many cycles.

We first record some bounds on the moments of the number of $(a,b)$-paths, beginning with the expected number of shortcutted paths. These moment bounds will be used to show that the two-sided tree construction can be coupled to a supercritical branching process.

\begin{lemma}\label{lmm:shortcutted_paths}
Conditional on $H^*=h$ where $h$ has at most $(\log_2 n)^2$ cycles, given a fixed vertex $v_0$ and  constants $a,b \geq 1$ the expected number of shortcutted $(a,b)$-paths starting from $v_0$ is $O(\log^2(n)/n).$ 
\end{lemma}
\begin{proof}
Given an $(a,b)$-path $v_0,...,v_{a+b}$ and another path $v'_0=v_0,...,v'_\ell=v_{a+b}$ with $\ell\le a+b$, choose the smallest $i$ such that $v'_i \ne v_i$ and the smallest $j\ge i$ such that $v'_j\in\{v_0,...,v_{a+b}\}$. Then $v'_{i-1},...,v'_j$ is a path of length at most $a+b$ such that $v'_{i-1}, v'_j\in\{v_0,...,v_{a+b}\}$ and this path does not share any other vertices or edges with $v_0,...,v_{a+b}$. So, it suffices to bound the expected number of $(a,b)$-paths that have such a side path.

Let $G'$ be the graph consisting of all the vertices and edges in the paths  
$v_0,...,v_{a+b}$ and $v'_{i-1},...,v'_j$, where the vertex labels are to be specified later.  Note that there are at most $(a+b+1)^3$ possible choices of $i$, $j$, and the value of $j'$ for which $v'_j=v_{j'}$,  and at most $2^{2a+2b}$ options for which edges in $G'$ are planted. 
Given $i$, $j$, $j'$, and the choice for the set of planted edges in $G'$, let $G'_r$ denote the subgraph of $G'$ formed by deleting all the unplanted edges. We call a connected component of $G'_r$ (which could be a single vertex)  a \emph{planted} component of $G'$; let $c$ be the number of planted components of $G'$. To determine $G'$, it remains to label the vertices in all its planted components. Because $G'$ consists of a path with a side path, there is at most one cyclic planted component, and all the other planted components must be paths. We consider the following two cases separately, depending on whether there exists a cyclic planted component.

\ul{Case 1:} There is a cyclic planted component $\mathcal{C}$. In this case, $\mathcal{C}$ must be a cycle in $H^*$. Since $H^*$ contains at most $(\log_2 n)^2$ cycles, there are at most $(\log_2 n)^2$ different choices for $\mathcal{C}$ and at most $2(2a+2b)$ options for which of $\mathcal{C}$'s vertices are $v_{i-1}$ and $v_{j'}$. Each other planted component must be a path contained in a cycle of $H^*$. 
The length of the path has been pre-determined when we specify which edges in $G'$ are planted. Thus, to determine the planted path component, it remains to specify the starting vertex of the path and which direction to traverse in a cycle of $H^*$. Therefore, each planted component other than $\mathcal{C}$ has at most $2n$ choices.
Moreover, note that $v_0$ is not contained in $\mathcal{C}$, because $(v_0,v_1)$ is required to be unplanted by the definition of an $(a,b)$-path. Therefore, the planted component containing $v_0$ is a singleton. Since $v_0$ has been pre-fixed, there are at most $2$ different choices for the planted component containing $v_0$. 
Let $m_r$ and $m_b$ denote
the number of planted and unplanted edges in $G'$, respectively. Recalling that $c$ denotes the number of planted components in $G'$, we have
$$
m_r + m_b = |V(G')| = |V(G'_r)| = m_r + c-1,
$$
where the first equality holds because $G'$ contains a single cycle, and the last equality holds because all the connected components of $G'_r$ are paths except for one being a cycle. It follows that $m_b=c-1.$

\ul{Case 2:} All planted components are paths. Since $v_0$ has been pre-fixed, there are at most $2$ different choices for the planted component containing $v_0.$ All the other planted components have at most $2n$ choices. Moreover, 
 we have
$$
m_r + m_b = |V(G')| = |V(G'_r)| = m_r + c.
$$
It follows that $m_b=c.$

Since every unplanted edge exists in the observed graph with probability $\lambda/n$, combining the above two cases, we deduce that the expected number of $(a,b)$-paths that have such a side path is at most 
$$
(a+b+1)^3 \times 2^{2a+2b} \times
\sum_{c=1}^{2a+2b}
\left( 2(2a+2b)(\log_2 n)^2 \times  (2n)^{c-2} \times \left(\frac{\lambda}{n} \right)^{c-1} + 2\times (2n)^{c-1} \left(\frac{\lambda}{n} \right)^{c}\right) =O\left( \log^2(n)/n\right). 
$$

\end{proof}

The following lemma gives an almost matching lower bound to Lemma \ref{validWalkCountLem}, even when a sufficiently small fraction of vertices is excluded and we are restricting to non-shortcutted paths. 
\begin{lemma}\label{lmm:lower_bound_path}
Let $G\sim \calG(n,\lambda,\delta)$ and  $H^*$ be the hidden planted $2$-factor on the set of $\delta n$ vertices.
Let $B$ be a subset of the vertices and $v \in B^c$ that are chosen in such a way that $|B|\le \zeta n$ for some $\zeta$, and conditional on $H^*, B, v$, the unplanted edges between vertices not in $B$ are still i.i.d.\ $\Bern(\lambda/n).$  Consider the subgraph of $G$ induced by $B^c$, denoted by $G[B^c]$.  
Then conditional on $B$, $v$, and $H^*=h$ where $h$ has at most $(\log_2 n)^2$ cycles, the expected number of non-shortcutted $(a,b)$-paths 
in $G[B^c]$ starting from $v$  is
at least $\left( 1- \frac{a(a+1)\zeta}{\delta } - \frac{(b-1)\zeta}{1-\delta  } - o(1) \right)
c_{a,b}
$
for any fixed constants $a,b \ge 1,$ where $\frac{b-1}{1-\delta}=0$ when $b=1$ and $\delta=1.$
\end{lemma}

\begin{proof}
In the following proof, we condition on $H^*=h$ where $h$ has at most $(\log_2 n)^2$ cycles. Given $\vec{a}, \vec{b} \in \naturals_+^k$, let $\calW^*_{\vec{a}, \vec{b}}(v)$ denote the set of $(\vec{a},\vec{b})$-paths 
starting from $v \notin B$ in the complete graph $K_n$ that do not go through any vertex in $B$ (where edges in the complete graph are red if they are contained in $H^*$, and they are blue otherwise). We will establish that
\begin{align}
\left| \calW^*_{\vec{a}, \vec{b}}(v)\right| \ge \left(1-\frac{(a+1)\zeta}{\delta}-o(1) \right)^k\left(1- \frac{\zeta}{1-\delta } - o(1)  \right)^{b-k}\cdot (2\delta n)^{k} \left((1-\delta)n \right)^{b-k}. \label{eq:W_a_b}
\end{align}
First, we enumerate sets of directed paths in $H^{\ast}$ which do not intersect with each other or $B$. 
For any vertex $v''\in H^*$ and any $\ell$, there exist two distinct paths starting at $v''$ and consisting of $\ell$ planted edges, unless $v''$ is contained in a cycle of length at most $\ell$. Since there are at most $(\log_2 n)^2$ cycles in $H^*$, the total number of length-$\ell$ directed planted paths in $H^*$ is at least $2\left( \delta n-\ell (\log_2 n)^2\right)$. Also, given a path of length $\ell'$ in $H^*$, there are at most $2(\ell'+\ell+1)$ directed paths of length $\ell$ in $H^*$ that intersect it, and there are at most $2(\ell+1)$ directed paths of length $\ell$ in $H^*$ that go through any given point. So, for any positive $(a_1,\ldots,a_k)$ with sum $a$, any $k'\le k$ and any choice of paths of lengths $a_1, a_2, \ldots, a_{k'-1}$ in $H^*$, the number of directed paths of length $a_{k'}$ in $H^*$ that do not intersect $v$, any of the first $k'-1$ paths, or any vertex in $B$ is at least 
\begin{align*}
&2\delta n-2a_{k'} (\log_2 n)^2 -(|B|+1)\cdot 2(a_{k'}+1)-\sum_{i=1}^{k'-1} 2(a_{k'}+a_i+1)\\
&\geq 2\delta n-2a_{k'} (\log_2 n)^2-2(k'+\zeta n) ( a_{k'}+1)-2\sum_{i=1}^{k'-1} a_i\\
&\ge 2\delta n-2a (\log_2 n)^2-2(k+\zeta n) (a+1) -2a.
\end{align*}
Therefore, the number of ways to select directed paths of lengths $a_1, \dots, a_k$ in $H^*$, none of which intersect or contain $v$, is at least  
$$
\left(2\delta n-2a (\log_2 n)^2-2(k+\zeta n) (a+1) -2a \right)^{k}=\left(1-(a+1)\zeta /\delta-o(1) \right)^k\cdot (2\delta n)^{k}.
$$
Given such a set of paths and any $(b_1,...,b_k)$ with sum $b$, there are at least 
$$
\prod_{i=0}^{b-k-1}
(n-\delta n -\zeta n-i) 
%
=
\left(1- \frac{\zeta}{1-\delta } - o(1)  \right)^{b-k} \left( (1-\delta) n \right)^{b-k}
$$
ways to choose an ordered set of $b-k$ vertices not in $H^*\cup B$. Putting $b_i-1$ of these vertices between the $i$th and $(i+1)$th planted paths for each $i$ always yields an $((a_1,...,a_k),(b_1,...,b_k))$-path starting from $v$  in the complete graph, thus establishing~\prettyref{eq:W_a_b}.

Each of these paths appears in $G$ with probability $(\lambda/n)^{b}$. 
For any given $a \ge 1$, there are $\binom{a-1}{k-1}$ distinct sequences of positive integers $(a_1,...,a_k)$ which sum to $a$, and similarly for $b$.
Thus, the expected number of $(a,b)$-paths from $v$ to $v'$ in $G$ avoiding $B$ is at least
\begin{align*}
&\sum_{k=1}^\infty \sum_{\vec{a},\vec{b} \in \naturals_+^k:\|\vec{a}\|_1=a,\|\vec{b}\|_1=b}  \left| \calW^*_{\vec{a}, \vec{b}}(v)\right| \cdot (\lambda/n)^{b} \\
& \ge \left(1- \frac{(a+1)\zeta }{\delta }-o(1) \right)^a \left(1- \frac{\zeta}{1-\delta} - o(1)  \right)^{b-1} \sum_{k=1}^\infty \binom{a-1}{k-1}\binom{b-1}{k-1} \cdot  (2\delta/(1-\delta))^k (\lambda(1-\delta))^b \\
& \ge 
\left( 1- \frac{a(a+1)\zeta}{\delta } - \frac{(b-1)\zeta }{1-\delta} - o(1) \right)
c_{a,b}.
\end{align*}

Finally, we apply~\prettyref{lmm:shortcutted_paths} to upper-bound the
 expected number of shortcutted $(a,b)$-paths in $G[B^c]$. Let $G'$ be $G$ with all vertex pairs $\{ (u,u') \notin H^*: u \in B \text{ or } u' \in B\}$ resampled according to i.i.d. $\Bern(\lambda/n)$. By~\prettyref{lmm:shortcutted_paths}, the
 expected number of shortcutted $(a,b)$-paths starting from $v$ in $G'$ is $O(\log^2 (n)/n).$ Since $G[B^c]=G'[B^c] \subset G'$, the expected number of shortcutted $(a,b)$-paths starting from $v$ in $G[B^c]$ is $O(\log^2 (n)/n).$ 
As a consequence, 
the expected number of non-shortcutted $(a,b)$-paths in $G[B^c]$ is at least $$
\left( 1- \frac{a(a+1)\zeta}{\delta } - \frac{(b-1)\zeta }{1-\delta} - o(1) \right)
c_{a,b}- O(\log^2 (n)/n)
= \left( 1- \frac{a(a+1)\zeta}{\delta } - \frac{(b-1)\zeta }{1-\delta} - o(1) \right)
c_{a,b}.
$$
\end{proof}

The following lemma gives an upper bound on the variance of the number of $(a,b)$-paths, which will be used in the tree construction to prove the impossibility of almost exact recovery.
\begin{lemma}\label{lemma:variance-paths}
 Fix $a,b \in \mathbb{Z}_+$ and a vertex $v \in V(H^{\ast})$. 
Let $X_{a,b}(v)$ denote the number of $(a,b)$-paths starting from $v$ in $G.$ Then $\mathbb{E}[X_{a,b}^2 \mid H^*] \leq C$ for some constant $C$ that depends only on $\lambda$, $a$, and $b$, and is independent of $n$. 
\end{lemma}

\begin{proof} 

Let $\calW_{a,b}(v)$ denote the set of $(a,b)$-paths that start from $v \in V(H^*)$ in the complete graph $K_n$ (where edges in the complete graph are red if they are contained in $H^*$, and they are blue otherwise). Then
$$
X_{a,b}=\sum_{P \in \calW_{a,b}(v)}  \indc{P \subset G}.
$$
By definition, 
\begin{align*}
\expect{X^2_{a,b} \mid H^* }
= 
\sum_{P, P' \in \calW_{a,b}(v) }
\prob{ P \subset G, P' \subset G \mid H^*}
 =\sum_{P, P' \in \calW_{a,b}(v)} 
 \left(\frac{\lambda}{n}\right)^{2b-|P_b \cap P_b'|} , 
\end{align*}
where $P_b$ denotes the blue (unplanted) segments of path $P$. 
Fixing a path $P \in \calW_{a,b}(v)$,   the number of paths $P' \in \calW_{a,b}$ for which 
$|P_b \cap P_b'|=e$ is at most
$$
\sum_{k>0,(a_1,...,a_k),(b_1,...,b_k):\sum a_i=a,\sum b_i=b} \binom{b}{e} \binom{b}{e} 2^e n^{b-e} 2^k,
$$
where the first $\binom{b}{e}$ counts the different choices of $e$ overlapping blue edges from $P$; the second $\binom{b}{e}
$ counts the different locations to put those $e$ overlapping blue edges when we list the blue edges of $P'$ from left to right starting from vertex $v$;  $2^e$ accounts for the different directions to put those $e$ overlapping blue edges; 
$n^{b-e}$ counts the vertex labels of the right endpoints of the additional blue edges of $P'$ after fixing the $e$ overlapping edges; finally, to further determine the red edges and the remaining vertex labels, we just need to follow the sequences $(a_1, \ldots, a_k)$
and $(b_1, \ldots, b_k)$, traverse a cycle of the $2$-factor $H^*$ for a direction, insert the red edges accordingly, and there are at most $2^k$ different directions to choose from. 

Moreover, recall from~\prettyref{eq:W_a_b_bound} that $|\calW_{a,b}(v)| \le C'_{a,b} n^{b}$ for some constant $C'$ that only depends on $a,b$. It follows that 
$$
\expect{X^2_{a,b} \mid H^*} 
\le C''_{a,b} 
\sum_{e=0}^{b} n^b n^{b-e} \left(\frac{\lambda}{n}\right)^{2b-e}
=C_{a,b,\lambda},
$$
where $C_{a,b,\lambda}$ is some constant that only depends on $a, b, \lambda. $
\end{proof}

We now demonstrate that sufficiently many trees are constructed. As a first step, we establish a lower bound on the number of available vertices throughout the execution of Algorithm \ref{alg:tree-construction2}.

\begin{lemma}\label{lmm:tree.facts2}
Let $t > 0$, $\gamma \in (0,1)$, and $\ell \in \mathbb{N}$ be constants. Conditioned on $H^*$, throughout the first $t$ iterations of tree construction in Algorithm \ref{alg:tree-construction2}, the number of available vertices satisfies $|\mathcal{A}| \geq n - 2\gamma n - 6\ell (2\lambda+4)^{2m^\ast}t$ with probability $1 - e^{-\Theta(t) }$.
\end{lemma}
\begin{proof}
Recall that initially $|\mathcal{A}|=(1-2\gamma)n.$
Let $\mathcal{A}_i$ be the set of available vertices after the $i^{\text{th}}$ pruning stage when Line \ref{line:remove1} executes, which happens whenever 
we find the $(m^{\ast}, m^{\ast})$-paths from a given hub vertex. Here $i \in \{0, \dots, 4\ell t\}$, since each tree has two sides and each side executes at most $2\ell$ pruning stages. We let $\mathcal{A}_0$ contain all vertices apart from the reserved vertices, and let $\mathcal{A}_{i} = \mathcal{A}_T$ for all $i > T$ where $T$ is the (random) number of pruning stages. Let $X_i$ be the number of vertices removed during the $i^{\text{th}}$ pruning stage, so that $X_i = |\mathcal{A}_{i-1}| - |\mathcal{A}_i|$. In order to bound the $X_i$'s, we introduce a sequence of independent branching processes $B_i$ with offspring distribution $2 +\Binom(n,p\triangleq \lambda/n)$. Let $Y_i$ be the number of nodes among the first $2m^{\ast} + 1$ levels of the $i^{\text{th}}$ branching process $B_i$. We claim that $X_i$ is stochastically dominated by $Y_i$.

To prove the claim, we condition on the realization of $H^{\ast}$. If $(u,v)$ is not an edge in $H^{\ast}$, then we call $(u,v)$ a potential blue edge. We call a potential blue edge $(u,v)$ ``unseen''
after the $i^{\text{th}}$ pruning stage if at least one of its endpoints is available; that is, 
$|\{u,v\} \cap \mathcal{A}_{i}| \geq 1$. 
 Let $U_i$ be the set of unseen potential blue edges after the $i^{\text{th}}$ pruning stage. 
Equivalently, $U_i$ contains all potential blue edges outside of the graph induced by $[n] \setminus \mathcal{A}_{i}$. 
Then conditioned on the graph induced by $[n] \setminus \mathcal{A}_{i}$, the edges within $U_i$ exist independently with probability $\lambda/n.$  For each $i \ge 1,$ $X_i$ is equal to the total number of nodes in the 
$2m^{\ast}$-neighborhood of a hub vertex $u$ at stage $i$ within $\mathcal{A}_{i-1}$, which is completely determined by $U_{i-1}$, Thus conditional 
on the graph induced by $[n] \setminus \mathcal{A}_{i-1}$,
the 
$2m^{\ast}$-neighborhood of $u$ within $\calA_{i-1}$ is stochastically dominated by the $i^{\text{th}}$ branching process $B_i$.

Having established the branching process domination, it suffices to bound $Y \triangleq \sum_{i=1}^{4\ell t } Y_i$, which we accomplish with a Chernoff bound strategy. For any $r,s > 0$, the Markov inequality yields
\begin{align}
\mathbb{P}\left(Y \geq s \right) = \mathbb{P}\left(e^{rY} \geq e^{rs} \right) \leq e^{-rs} \mathbb{E}\left[e^{rY}\right] =e^{-rs} \left(\mathbb{E}\left[e^{rY_1}\right]\right)^{4\ell t}. \label{eq:tree-Chernoff}
\end{align}
To compute the MGF of $Y_1$, 
let $Z_j$ be the total progeny of a depth-$j$ branching process with offspring distribution $2 + \Binom(n,p)$. We claim that $Z_{j+1}$ is equal in distribution to
\[1+ \sum_{i=1}^{W} V_{j,i}, \quad V_{j,i} \iiddistr Z_j, \quad  W \sim 2+\Binom(n,p),\]
where $W$ is independent of the $V_{j,i}$'s.
To see this, observe that we can find the number of offspring through level $j+1$ by adding up the number of descendants through level $j$ relative to each of the children of the root, and then adding $1$ for the root itself. Each offspring of the root has a number of descendants through level $j$ (including itself) with distribution $Z_j$.
Thus, the MGF of $Z_{j+1}$ satisfies the following recurrence relation:  
\begin{align} 
\mathbb{E}\left[e^{rZ_{j+1 }}\right] &= \mathbb{E}\left[e^{r \left(1 + \sum_{i=1}^W V_{j,i}\right)} \right]\\
&= e^r \expect{ \left(\mathbb{E}\left[e^{rZ_{j}} \right] \right)^W} \nonumber \\
&=e^r\left(\mathbb{E}\left[e^{rZ_{j}} \right] \right)^{2}\left(1-p +p \cdot \mathbb{E}\left[e^{r Z_{j}} \right]\right)^n \nonumber \\
&\leq e^r \left(\mathbb{E}\left[e^{rZ_{j}}\right] \right)^{2} \left(\exp\left(-p + p \mathbb{E}\left[e^{r Z_{j}} \right] \right)\right)^n \nonumber \\
&= \left(\mathbb{E}\left[e^{rZ_{j}}\right] \right)^{2} \exp\left(r-\lambda + \lambda \mathbb{E}\left[e^{r Z_{j}} \right] \right). \label{eq:recurrence}
\end{align}
Fix any  $\beta>0$ such that $\alpha \triangleq \frac{e^{\beta}-1}{\beta} \geq \frac{3}{2}$. We prove via induction that 
\begin{equation}
\mathbb{E}\left[e^{r Z_j} \right]
\le \exp \left( [\alpha(\lambda+2)]^j r\right),  \quad
\forall j \in [2m^*], \forall 0 \le r \le \beta/[\alpha(\lambda+2)]^{2m^*}  \label{eq:inductive-hypothesis}.
\end{equation}
Since $Z_0 = 1$, the base case holds trivially. We next assume the inductive hypothesis \eqref{eq:inductive-hypothesis} for some $j \in \{1, 2, \dots, 2m^{\ast} - 1\}$. Using the recurrence relation \eqref{eq:recurrence}, we have for all $ 0 \le r \le \beta/[\alpha(\lambda+2)]^{2m^*} $,
\begin{align}
    \mathbb{E}\left[e^{rZ_{j+1}}\right] &\leq \exp \left\{2\left[\alpha (\lambda + 2) \right]^j r + r- \lambda + \lambda e^{[\alpha (\lambda + 2)]^j r} \right\} \nonumber\\
&\leq \exp \left\{2\left[\alpha (\lambda + 2) \right]^j r + r- \lambda + \lambda \left( 
1+ \alpha [\alpha (\lambda + 2)]^j r \right)\right\} \nonumber \\
& = \exp \left\{ [\alpha (\lambda + 2)]^j r 
 \left( 2+ \lambda \alpha \right) + r \right\}, \label{eq:recurrence-intermediate}
\end{align}
where the first inequality follows by plugging  the inductive hypothesis for $j$, and the second inequality holds due to $e^r \le 1+\alpha r$ for all $0 \le r \le \beta$ and the fact that $0 \le [\alpha (\lambda + 2)]^j r \le \beta$.

Recalling that $\alpha \geq \frac{3}{2}$, we have
\[ 2r(1-\alpha) [\alpha (\lambda + 2)]^j + r \leq 2r(1-\alpha)+r = r(3-2\alpha) \leq 0,\]
so that
\[ 2r(1-\alpha) [\alpha (\lambda + 2)]^j + r \leq 0.\]
Rearranging, the above inequality is equivalent to
\begin{align}
    [\alpha (\lambda + 2)]^j r 
 \left( 2+ \lambda \alpha \right) + r \leq [\alpha (\lambda + 2)]^{j+1} r \label{eq:recurrence-intermediate-2}.
\end{align}
Substituting \eqref{eq:recurrence-intermediate-2} into \eqref{eq:recurrence-intermediate}, we obtain
\[\mathbb{E}\left[e^{rZ_{j+1}}\right] \leq \exp\left\{[\alpha (\lambda + 2)]^{j+1} r \right\},\]
thus establishing \eqref{eq:inductive-hypothesis}.

Having established \eqref{eq:inductive-hypothesis} for all $j \in \{1,2, \dots, 2m^{\ast}\}$, we get that 
$$
\mathbb{E}\left[e^{rY_1}\right] = \mathbb{E}\left[e^{rZ_{2m^{\ast} }}\right]
\leq \exp \left( [\alpha(\lambda+2)]^{2m^{\ast}} r\right), ~~~\forall 0 \le r \le \beta/[\alpha(\lambda+2)]^{2m^*}.  
$$
Substituting into \eqref{eq:tree-Chernoff}, we get that 
\begin{align*}
\mathbb{P}\left(Y \geq 6 \ell t [\alpha(\lambda + 2)]^{2 m^{\ast}} \right) &\leq \exp\left(-6r \ell t [\alpha(\lambda + 2)]^{2 m^{\ast}} + 4\ell t  [\alpha (\lambda + 2)]^{2m^{\ast}} r \right)\\
&=\exp\left(-2r \ell t [\alpha(\lambda + 2)]^{2 m^{\ast}} \right) = \exp \left( -2 \ell t \beta\right),
\end{align*}
where the last equality holds by choosing $r=\beta/[\alpha(\lambda+2)]^{2m^*}$. Finally, we arrive at the desired conclusion by picking $\beta=1$ and noting that  $\alpha\le 2.$
\end{proof}

Next, we will show that Algorithm \ref{alg:tree-construction2} results in enough trees for the purposes of creating balanced cycles. 
\begin{lemma}\label{lemma:enough-trees2}
Let $\delta, \lambda > 0$ satisfy $\lambda>\frac{1}{(\sqrt{2\delta}+\sqrt{1-\delta})^2}$ and condition on $H^*$ containing at most $(\log_2 n)^2$ cycles. Let $m^{\ast}$ be the value from Lemma \ref{lmm:m-star} with $c_{m^*, m^*} >1$, where $m^*=1$ when $\delta=1.$
Let $\gamma > 0$ be a small constant such that 
\begin{align}
\zeta \triangleq 2 \gamma + 6 \gamma(2\lambda + 4)^{2m^*} < \frac{c_{m^*, m^*} - 1 }{2 c_{m^*, m^*} } \cdot  \frac{1}{m^*(m^*+1)/\delta + (m^*-1)/(1-\delta)} . \label{eq:zeta}
\end{align}


Then  for all $\ell \in \mathbb{N}$,
with probability $1-e^{-\Omega(n)}$, Algorithm \ref{alg:tree-construction2} results in at least $K_1 = C' K$ two-sided trees, where $K = \frac{\gamma n}{\ell}$ as in Line \ref{line:K}, and $C' > 0$ is a function of $\delta$ and $\lambda$. Specifically, $C' = p^2/2$ with $p = \frac{\mu^2-\mu}{\mu^2-\mu+ C +1/4}$, where $\mu=(c_{m^*,m^*}+1)/2$ and $C$ is the constant from Lemma~\ref{lemma:variance-paths}.
\end{lemma}

\begin{proof}
Algorithm~\ref{alg:tree-construction2} iterates through $t=1,\ldots K$; at iteration $t$, it successfully outputs a two-sided tree if both sides of the tree grow to contain at least $2\ell$ hub nodes. Ideally, we would like to argue that at iteration $t$, conditional on the history of the algorithm, the probability of success is lower-bounded by some constant, hence a constant proportion of the $K$ iterations succeed with high probability. However, this argument does not go through given any history of the algorithm. In particular, if the algorithm explores too many vertices before the $t^{\text{th}}$ iteration, there may not be enough remaining vertices to grow large enough trees. To resolve this, we use Proposition~\ref{lmm:tree.facts2} to argue that with high probability, there will be enough vertices available for the algorithm to explore throughout the entire construction. Below is the precise argument.

Let $E$ be the event that throughout all $K=\gamma n/\ell$ iterations of Algorithm~\ref{alg:tree-construction2}, the number of available vertices is at least $(1-\zeta) n$. Since by assumption $1-\zeta > 1 - \frac{\delta}{2}$, the event $E$ implies that there exists at least one planted edge among the available vertices and hence the $\texttt{FAIL}$ condition in Line \ref{line:fail} is not triggered.
By Proposition~\ref{lmm:tree.facts2}, $\mathbb{P}(E)=1-e^{-\Theta(n)}$.

Let $S_{2t-1}$ (resp. $S_{2t}$) denote the events that the left (resp. right) side of the $t^{\text{th}}$ two-sided tree grows to contain at least $2\ell$ hub nodes. We couple the growth on each side of this two-sided tree with an auxiliary tree construction process which mimics the true one until the event $E$ fails. More precisely, whenever we are about to add the offspring to a vertex in the tree, we check whether the number of available vertices is at least $(1-\zeta)n$. If so, we directly copy the number of offspring to the auxiliary process. If not, the number of offspring added to the auxiliary process is instead sampled  from any distribution supported on $\mathbb{N}$ with mean $\mu$ and variance at most $C$, where $\mu=(c_{m^*,m^*}+1)/2$ and $C$ is the constant from Lemma~\ref{lemma:variance-paths}. Note that $\mu$ and $C$ only depend on $\delta$ and $\lambda$.

Let $\tilde{S}_{2t-1}$ and $\tilde{S}_{2t}$ denote the corresponding events for the coupling process. We have
\begin{align}
\mathbb{P}\left\{\sum_{t\leq K} \mathbbm{1}\{S_{2t-1} \cap S_{2t}\} < K_1\right\}
\nonumber &\leq  \mathbb{P}\left(\left\{\sum_{t\leq K} \mathbbm{1}\{\tilde{S}_{2t-1} \cap \tilde{S}_{2t}\} < K_1\right\}\cap E\right) + \mathbb{P}(E^c)\\
\label{eq:prob.many.trees} &\leq \mathbb{P}\left(\sum_{t\leq K} \mathbbm{1}\{\tilde{S}_{2t-1} \cap \tilde{S}_{2t}\} < K_1\right) + \mathbb{P}(E^c).
\end{align}
The first inequality holds because on event $E$, $S=\tilde{S}.$
Next, we lower-bound the survival probability of the auxiliary trees given the history of the algorithm. For simplicity, let us first focus on the growth of the left tree. Let $\mathcal{F}_{2t-2}$ denote the history of the auxiliary trees up to iteration $t$. We will show that $\mathbb{P}(\tilde{S}_{2t-1}|\mathcal{F}_{2t-2})\geq p$, where $p = \frac{\mu^2-\mu}{\mu^2-\mu+C+1/4}$. 

 To see this, denote $Z_i$ as the number of offspring of the $i^{th}$ node in breadth-first order in the left auxiliary tree. Let $\mathcal{F}_{2t-2,i-1}$ be the union of $\mathcal{F}_{2t-2}$ and the history of the auxiliary process up to before $Z_i$ is generated. Let $E_{i}$ be the event that there are at least $cn$ available vertices right before $Z_i$ is generated. Then 
\[
\mathbb{E}\left(Z_i | \mathcal{F}_{2t-2,i-1}\right) \geq \min\left\{\mathbb{E}\left(Z_i | \mathcal{F}_{2t-2,i-1}, E_i\right), \mathbb{E}\left(Z_i | \mathcal{F}_{2t-2,i-1}, E_i^c\right)\right\}.
\]
On $E_i$, the growth of the auxiliary tree and the original tree coincide exactly up to (and including) $Z_i$. By Lemma~\ref{lmm:lower_bound_path}, we have that conditional on $\mathcal{F}_{2t-2,i-1}$ and $E_i$, the expected number of non-shortcutted $(m^*,m^*)$-paths that originate from the $i^{\text{th}}$ hub node and only contain available vertices is at least 
\[ 
\left( 1- \frac{m^*(m^*+1)\zeta}{\delta } - \frac{(m^*-1)\zeta}{1-\delta} - o(1) \right) c_{m^*,m^*}.
\] 
By the choice of $\zeta$ as per~\prettyref{eq:zeta}, we have
$$
\left( 1- \frac{m^*(m^*+1)\zeta}{\delta } - \frac{(m^*-1)\zeta}{1-\delta}  \right) c_{m^*,m^*}
> \left[1 - \frac{1}{2} \cdot \frac{c_{m^*, m^*} -1}{c_{m^*, m^*}} \right]c_{m^*, m^*} = (c_{m^*,m^*}+1)/2 = \mu,
$$
which shows that $\mathbb{E}\left(Z_i | \mathcal{F}_{2t-2,i-1}, E_i\right)\geq \mu$ for $n$ sufficiently large.
On the other hand, $\mathbb{E}\left(Z_i | \mathcal{F}_{2t-2,i-1}, E_i^c\right)\geq \mu$ from the definition of the auxiliary process. Thus $\mathbb{E}\left(Z_i | \mathcal{F}_{2t-2,i-1}\right) \geq \mu$ for all $i$. 

Similarly, we can apply Lemma~\ref{lemma:variance-paths} to bound the second moment as follows. 
\[
\mathbb{E}\left((Z_i)^2 | \mathcal{F}_{2t-2,i-1}\right) \leq \max\left\{\mathbb{E}\left((Z_i)^2 | \mathcal{F}_{2t-2,i-1}, E_i\right), \mathbb{E}\left((Z_i)^2 | \mathcal{F}_{2t-2,i-1}, E_i^c\right)\right\} \leq C.
\]
The first and second moment control allows us to apply Lemma~\ref{lmm:branching}, which lower-bounds the survival probability of such a history-dependent branching process. The lemma ensures that
\[
\prob{\tilde{S}_{2t-1}|\mathcal{F}_{2t-2}} \geq p=\frac{\mu^2-\mu}{\mu^2-\mu+ C +1/4}.
\]
The same argument applies to the growth of the right auxiliary tree, so that
$
\prob{\tilde{S}_{2t}|\mathcal{F}_{2t-1}} \geq p,
$
where $\mathcal{F}_{2t-1}$ is the union of $\mathcal{F}_{2t-2}$ and growth of the left auxiliary tree (at iteration $t$). Thus
\begin{equation}
\label{eq:success.prob.t.tree}
\prob{\tilde{S}_{2t-1}\cap \tilde{S}_{2t}|\mathcal{F}_{2t-2}} = \prob{\tilde{S}_{2t-1}|\mathcal{F}_{2t-2}}\prob{\tilde{S}_{2t}|\mathcal{F}_{2t-1}} \geq p^2.
\end{equation}
From here, we can bound $\mathbb{P}(\sum_{i\leq K} \mathbbm{1}\{\tilde{S}_{2t-1}\} \mathbbm{1}\{\tilde{S}_{2t}\}<K_1)$ by coupling $\{\mathbbm{1}\{\tilde{S}_{2t-1}\cap \tilde{S}_{2t}\}\}_{t\leq K}$ with a sequence $\{X_t\}_{t\leq K}$ of i.i.d. $\text{Bern}(p^2)$ random variables. In more detail, when $\mathbbm{1}\{\tilde{S}_{2t-1} \cap \tilde{S}_{2t}\}=0$, let $X_t=0$; when $\mathbbm{1}\{\tilde{S}_{2t-1} \cap \tilde{S}_{2t}\}=1$, draw $X_t$ from the Bernoulli distribution with success probability $p^2/\mathbb{P}(\tilde{S}_{2t-1} \cap \tilde{S}_{2t-2}|\mathcal{F}_{2t})$ (Note that by~\eqref{eq:success.prob.t.tree} this is a valid probability). By construction, we have $X_t\leq \mathbbm{1}\{\tilde{S}_{2t-1} \cap \tilde{S}_{2t}\}$ for all $t$, and
\[
\mathbb{P}(X_t=1|X_1,\ldots,X_{t-1}) = \mathbb{P}\left(\tilde{S}_{2t-1} \cap \tilde{S}_{2t}|\mathcal{F}_{2t-2}\right)\cdot \frac{p^2}{\mathbb{P}\left(\tilde{S}_{2t-1} \cap \tilde{S}_{2t}|\mathcal{F}_{2t-2}\right)} = p^2.
\]
Thus the $X_t$'s are distributed i.i.d. $\text{Bern}(p^2)$, and
\[
\mathbb{P}\left(\sum_{t\leq K} \mathbbm{1}\{\tilde{S}_{2t-1} \cap \tilde{S}_{2t}\} < K_1\right) \leq \mathbb{P}\left(\sum_{t\leq K}X_t<K_1\right) = e^{-\Omega(n)}
\]
by choosing $K_1=C'K$ with constant $C'=p^2/2$ only depending on $\delta$ and $\lambda$. Combined with~\eqref{eq:prob.many.trees}, and the fact that $\mathbb{P}(E^c)=e^{-\Theta(n)}$, we have shown that
\[
\mathbb{P}\left\{\sum_{t\leq K} \mathbbm{1}\{S_{2t-1} \cap S_{2t}\} < K_1\right\} = e^{-\Theta(n)}.
\]
In other words, with probability $1-e^{-\Omega(n)}$, Algorithm~\ref{alg:tree-construction2} yields at least $K_1$ two-sided trees.

\end{proof}

\subsection{Cycle construction via sprinkling}\label{sec:sprinkling}

We now give the details of the preprocessing procedure, which reserves edges and creates the set of available vertices. The reserved edges will be used for the purposes of ``sprinkling'' to connect trees into cycles.

\begin{algorithm}[h]
\caption{Reserve Edges}\label{alg:matching}
\begin{algorithmic}[1]
\Statex{\bfseries Input:} Graph $H^*$ on $n$ vertices, $\gamma \in (0, \delta/5]$ 
\Statex{\bfseries Output:} A set $E^\ast$ of $\gamma n$ vertex-disjoint red edges of $G$, available vertices $\mathcal{A}$
\State Let $E^{\ast} = \emptyset$, $S = E(H^*)$.\;
\For{$i \in \{1, 2, \dots, \gamma n\}$}
        \State Choose an arbitrary edge $e$ from $S$ and add it to $E^{\ast} $.
        \State \label{line:remove-neighbors} Remove $e$ and all edges at distance at most $2$ from $S$ (that is, edges which have an endpoint that is connected by an edge to an endpoint of $e$).
\EndFor
\State Let $V_1$ be the set of endpoints of edges in $E^{\ast}$. 
Let $\mathcal{A} = [n] \setminus V_1$ be the set of \emph{available} vertices. 
\end{algorithmic}
\end{algorithm}
\FloatBarrier

Since Step 4 of \prettyref{alg:matching} removes at most five edges, the final set $E^*$ contains $\gamma n$ edges in total, provided that $\gamma \leq \delta/5$. Note that any set of parameters satisfying the conditions of Lemma \ref{lemma:enough-trees2} will always have $\gamma \leq \delta/5$, and will thus be suitable for this algorithm.


\begin{algorithm}
\caption{Cycle Construction} \label{alg:cycle-construction2}
\begin{algorithmic}[1]
\Statex{\bfseries Input:} Graph $G$ on $n$ vertices with red subgraph $H^{\ast}$ on $\delta n$ vertices, reserved edge parameter $\gamma$, path length parameter $m \in \mathbb{N}$, tree size parameter $\ell \in \mathbb{N}$, and degree parameter $d$
\Statex{\bfseries Output:} A set of alternating cycles $\mathcal{C}$ on $G$

\smallskip
\State 
Apply Algorithm \ref{alg:matching} to input $(H^{\ast}, \gamma)$, obtaining the set of $\gamma n$ reserved edges $E^{\ast}$ and the set of $(1-2\gamma) n$ available vertices $\mathcal{A}$. We assume $|E^*|$ is even.
\State \label{line:make-trees} Let $\mathcal{T} = (L_i, R_i)_{i=1}^{K_1}$ be the output of Algorithm \ref{alg:tree-construction2} on input $(G, 
\mathcal{A}, m, \ell)$.
\State Randomly partition $E^{\ast}$ into two equally-sized sets $(E^{\ast}_L, E^{\ast}_R)$. 
For each $(u,v) \in E^{\ast}$ with $u < v$, designate $u$ as the ``tree-facing'' vertex and designate $v$ as the ``linking'' vertex. 
\State Initialize an empty (bipartite) graph $\overline{G}$.
\For{$i \in [K_1]$}
\If{$L_i$ is blue-connected to at least $d$ unmarked tree-facing endpoints among $E_L^{\ast}$ and the same is true for $R_i$ with respect to $E_R^{\ast}$}\label{step:collision}
\State Let the first $d$ of these edges be denoted $\mathcal{E}(L_i) \subset E_L^{\ast}$ and $\mathcal{E}(R_i) \subset E_R^{\ast}$. 
\State Mark all edges among $\mathcal{E}(L_i) \cup \mathcal{E}(R_i)$.
\State Include $i$ as a vertex on both sides of $\overline{G}$, and connect them by a red edge.
\EndIf
\EndFor
\For{$i \in [K_1]$}
\For{$j \in [K_1]$}
\If{both $i$ and $j$ are vertices in $\overline{G}$, and some linking endpoint in $\mathcal{E}(L_i)$ is connected to some linking endpoint in $\mathcal{E}(R_j)$ by a blue edge}
\State Connect $i$ and $j$ by a blue edge in $\overline{G}$.
\EndIf
\EndFor
\EndFor
\item Find the set of alternating cycles in $\overline{G}$, and return their preimage in $G$. 
\end{algorithmic}
\end{algorithm}



Algorithm \ref{alg:cycle-construction2} connects the trees into cycles via reserved edges, similarly to Algorithm 3 in \cite{gaudio2025all}. At a high level, a cycle is formed by linking some right tree $R_{i_1}$ to a left tree $L_{i_2}$, then linking $R_{i_2}$ to $L_{i_3}$, and so on, back to $L_{i_1}$. In order to retain some independence for the purpose of linking, trees are connected to each other via a ``five edge construction:'' a hub node in $R_{i}$ connects to a reserved red edge in $E_R^{\ast}$ while a hub node in $L_{j}$ connects to a red edge in $E_L^{\ast}$; in turn, the red edges are connected by a blue edge (see Figure \ref{fig:five-edge-construction}). By design, the endpoints of the reserved edges can only be connected by blue edges, due to Algorithm \ref{alg:matching}, Line \ref{line:remove-neighbors}. 
Since the five edges comprise three blue edges and two red edges, and each two-sided tree has a central red edge, the overall construction preserves the balancedness of blue and red edges.
\begin{figure}[h]
    \centering
    \includegraphics[width=0.5\linewidth]{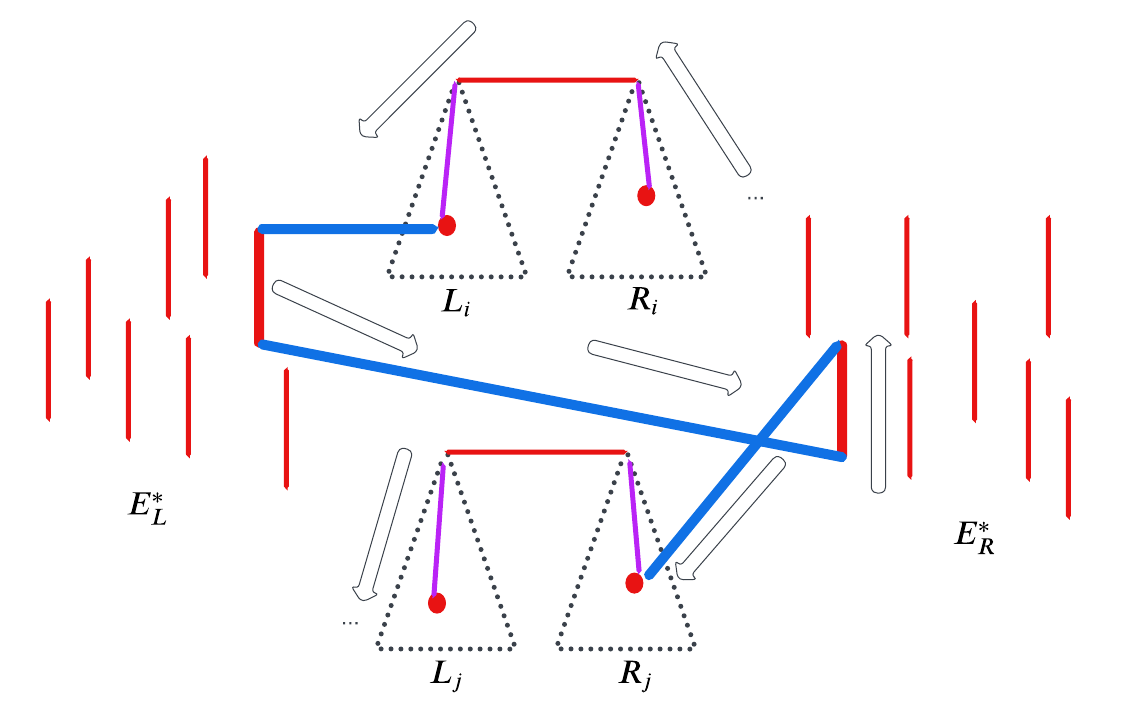}
    \caption{Five edge construction (reproduced from \cite{gaudio2025all}).}
    \label{fig:five-edge-construction}
\end{figure}

We finally show that Algorithm \ref{alg:cycle-construction2} produces many long, balanced cycles; Lemma \ref{lmm:Mbad-delta} then follows directly by observing that every balanced cycle of length $t$ corresponds to a distinct $2$-factor $H$ on $\delta n$ edges with $|H\Delta H^*|=t.$ 
\begin{lemma}\label{lemma:cycle-construction2}
Let $0<\delta\le 1$, $\lambda>\frac{1}{(\sqrt{2\delta}+\sqrt{1-\delta})^2}$, and condition on $H^*$ containing at most $(\log_2 n)^2$ cycles.
Let $\mathcal{C}$ be the output of Algorithm \ref{alg:cycle-construction2} on input $(G, \gamma, m, \ell, d)$, where $\gamma > 0$ satisfies the requirements in Lemma \ref{lemma:enough-trees2}, the value $m = m^{\ast}$ is as defined in Lemma \ref{lmm:m-star}, $\ell =  \frac{2^{14} \log(32 e)}{\lambda^2 \gamma^2} \alpha$, and $d = \frac{2^{11} \log (32 e)}{\lambda \gamma} \alpha$, where $\alpha \ge 1$ is a sufficiently large constant. Then there exist constants $c_1, c_2> 0$ such that $\mathcal{C}$ contains at least $e^{c_1 n}$ balanced cycles of length at least $c_2 n$, with probability $1 - e^{-\Omega(n)}$.
\end{lemma}

Lemma \ref{lemma:cycle-construction2} requires the following result of \cite{Ding2023}.
\begin{lemma}\cite[Lemma 7]{Ding2023}\label{lemma:bipartite-ER} 
Let $G$ be a bi-colored bipartite graph on $[k] \times [k]'$ whose $k$ red edges are defined by a perfect matching, and blue edges are generated from a bipartite Erd\H{o}s--R\'enyi graph with edge probability $\frac{D}{k}$. If $k \geq 525$ and $D \geq 256 \log(32 e)$, then with probability at least $1-\exp\left(-\frac{Dk}{2^{14}}\right)$, $G$ contains $\exp\left(\frac{k}{20}\right)$ distinct alternating cycles of length at least $\frac{3k}{4}$.
\end{lemma}

\begin{proof}[Proof of Lemma \ref{lemma:cycle-construction2}] The proof is analogous to the proof of \cite[Lemma C.6]{gaudio2025all}.
By Lemma \ref{lemma:enough-trees2}, at least $K_1 = C' K = C' \frac{\gamma n}{\ell}$ two-sided trees $\{(L_i, R_i)\}_{i=1}^{K_1}$ are produced, with probability $1-e^{-\Omega(n)}$. If $L_i, R_j$ each connect to at least $d$ edges among $E_L^{\ast}, E_R^{\ast}$ respectively, then $L_i$ and $R_j$ connect in the five-edge construction with probability at least 
\[1 - \left(1 - \frac{\lambda}{n}\right)^{d^2} \geq \frac{\lambda d^2}{2n}.\]

Let $c = \frac{512 \log (32e)}{\lambda d^2}$. We want to show that at least  $cn$ trees are such that both sides are connected to at least $d$ tree-facing endpoints (and we call such trees ``well-connected''). We use the fact that when identifying the tree-facing neighbors of a tree $(L_i, R_i)$ with $i \in [K_1]$, there are at least $\frac{\gamma n}{2}  - cnd$ tree-facing vertices that are not yet connected to a tree, as long as we have not yet found $cn$ well-connected trees. By our choice of $d$ and $\ell$, we have
\begin{align*}
d = \frac{2^{11} \log (32 e)}{\lambda \gamma} \alpha = \frac{\ell \lambda}{2} \cdot \frac{\gamma}{4}
 \le \frac{\ell \lambda}{2}\left(\frac{\gamma}{2} - \frac{512 \log (32e)}{\lambda d} \right)
= \frac{1}{2} \cdot \frac{\lambda}{n} \cdot \ell \cdot \left(\frac{\gamma n}{2} - cnd \right)
\end{align*}
Thus, $d$ is upper-bounded by half of the \emph{expected} number of tree-facing neighbors of $L_i$ or $R_i$, and the number of tree-facing neighbors of a given tree stochastically dominates the $\text{Bin}\left(\ell\left(\frac{\gamma n}{2} - cnd \right), \lambda/n\right)$ distribution, which in turn dominates the $\text{Bin}(2dn/\lambda, \lambda/n)$  distribution. Since the mean and median of any binomial random variable differ by at most $1$, we have that the median of the $\text{Bin}(2dn/\lambda, \lambda/n)$ distribution is at least $2d - 1 \geq d$. Thus, the probability that $L_i$ is connected to at least $d$ tree-facing endpoints is at least $1/2$, and the same is true for $R_i$. It follows that a particular tree is connected on both sides to $d$ tree-facing vertices with probability at least $\left(\frac{1}{2}\right)^2$.

Due to independence across trees and the Chernoff bound, we see that for any constant $\eta > 0$, there are at least $(1-\eta) \frac{K_1}{4}$ trees which are connected to at least $d$ tree-facing endpoints on both sides, with probability $1-e^{-\Omega(n)}$. Since both $d$ and $\ell$ are linear in $\alpha$, for sufficiently large constant $\alpha$,
\begin{align*}
    cn &= \frac{512 \log (32e)}{\lambda d^2} n \leq \frac{1-\eta}{8} \left(\frac{\mu^2-\mu}{\mu^2-\mu+ C +1/4}\right)^2 \cdot \frac{\gamma n}{\ell} = (1 - \eta) \frac{K_1}{4},
\end{align*}
where $\mu=(c_{m^*,m^*}+1)/2$ and $C$ is the constant from Lemma~\ref{lemma:variance-paths}. It follows that with probability $1-e^{-\Omega(n)}$, there are at least $cn$ trees which are connected to at least $d$ tree-facing endpoints on both sides.

Now we can couple to a bi-colored bipartite graph with $cn$ vertices on each side, representing $cn$ two-sided trees. The bipartite graph contains a perfect red matching, as well as independent blue edges which exist with probability $\frac{\lambda d^2}{2n}$. To apply Lemma \ref{lemma:bipartite-ER}, observe that the expected number of blue edges which are adjacent to a given vertex in the bipartite graph is $\frac{\lambda d^2}{2n} \cdot cn = 256 \log (32e).$ The conclusion follows from Lemma \ref{lemma:bipartite-ER}.
\end{proof}


\FloatBarrier
\section{Computationally efficient estimators}\label{sec:algorithm}


In this section, we design a computationally efficient estimator that attains the information-theoretic threshold for almost exact recovery. The estimator is given in Algorithm~\ref{alg:search}. Instead of searching for all cycles (which is computationally intractable), the algorithm builds a subgraph from short trails, combining or trimming them in a greedy manner to make the output subgraph as large as possible.

\begin{theorem}\label{thm:alg}
Consider the planted cycles model $\calG(n,\lambda,\delta)$ conditional on $H^*=h$, where $\delta \in (0,1]$. If $\lambda < \frac{1}{(\sqrt{2\delta} + \sqrt{1-\delta})^2}$, then Algorithm~\ref{alg:search}
outputs a set  $H$of edges such that $\expect{|H\Delta H^*|}=o(n)$, achieving almost exact recovery. Moreover, Algorithm \ref{alg:search}  runs in time $O\left(n^{3 + \log(2 +\lambda)} \right)$ with high probability.  
\end{theorem}
We prove the computational and statistical guarantees of Algorithm~\ref{alg:search} in the following Lemmas~\ref{lmm:decomposition_new}--\ref{lmm:poly.time}. In particular, \prettyref{thm:alg} follows by combining~\prettyref{lmm:alg.ae.recovery}, \prettyref{lmm:alg.Hgrows},
and~\prettyref{lmm:poly.time}. By design, the algorithm returns a subgraph  $H$ which consists of cycles and possibly a small number of paths. In Lemma~\ref{lmm:decomposition_new}, we first generalize Lemma~\ref{lemma:difference-graph} to show that the difference graph $H\Delta H^*$ can be decomposed into a collection of edge-disjoint trails with the desired alternating properties at vertices in $V(H)\cap V(H^*).$
Building upon such a decomposition, in Lemma~\ref{lmm:alg.ae.recovery}, we show that any such subgraph $H$ attains almost exact recovery of the true planted cycles $H^*$ whenever information-theoretically possible, as long as $H$ is nearly as large as $H^*$ and does not contain too many paths. We then prove in Lemma~\ref{lmm:alg.Hgrows} that Algorithm~\ref{alg:search} will indeed not output a large number of paths, and that through its greedy updates, the size of the output subgraph must be at least almost as large as $H^*$. Lemmas~\ref{lmm:alg.ae.recovery},~\ref{lmm:alg.Hgrows} combined imply that Algorithm~\ref{alg:search} attains the information-theoretic threshold for almost exact recovery. Finally, in Lemma~\ref{lmm:poly.time}, we prove that with high probability, Algorithm~\ref{alg:search} runs in polynomial time. 

Before stating and proving the lemmas, below is a brief description of Algorithm~\ref{alg:search} and the intuition behind its design. From an empty subgraph $H$, the algorithm searches through the collection of short trails to find ``good'' candidates $P$ to update $H$ to $H\Delta P$ (i.e. the XOR operation),  while making sure that the updated $H$ remains a collection of vertex-disjoint cycles and paths. That is, no vertex has degree larger than $2$ after the update. To clarify what qualifies a trail $P$ to be ``good,'' recall that the algorithm aims to construct an $H$ that 1) contains many edges, and 2) contains few paths (i.e. few vertices of degree $1$). Accordingly, two types of trails are considered ``good'':

\begin{enumerate}
\item ``Cost-free'' trails, which when XOR'ed onto $H$, strictly increases $|H|$ (number of edges) without increasing the number of degree-$1$ vertices.

\item ``Cost-effective'' trails, which when XOR'ed onto $H$, increase the number of degree-$1$ vertices by at most $2$, but increase $|H|$ by at least $\sqrt{\log n}$.
\end{enumerate}

In each iteration, the greedy algorithm invokes subroutines A and B to search for ``good'' trails of these two types. The algorithm terminates only when neither subroutine is able to find a suitable trail to update $H$.

\definecolor{ao(english)}{rgb}{0.0, 0.5, 0.0}

\begin{algorithm}[]
\caption{Approximate algorithm for finding a disjoint union of cycles }\label{alg:search}
\begin{algorithmic}[1]
\Statex{\bfseries Input:} A graph $G$
\Statex{\bfseries Output:} 
A vertex-disjoint union of cycles and paths, 
$H$
\smallskip
\State Let $n$ be the number of vertices in $G$.
\State Let $H=\emptyset$, boolean $can\_grow=True$.
\State Let $S$ be the set of all trails (closed or open) in $G$  of length less than $\log n$.
\While{$can\_grow$}
 \Statex{\;\;\;\; \textcolor{ao(english)}{
 \emph{// Subroutine A: search for ``cost-free'' candidate trails}}}
\For{$P\in S$} 
\If{$|H\Delta P|>|H|$, there is no vertex of degree greater than $2$ in $H\Delta P$, and there are at least as many degree $1$ vertices in $H$ as in $H\Delta P$,} \label{alg.step.find.short.P}
\State Set $H=H\Delta P$
\EndIf
\EndFor
\Statex{\;\;\;\;\textcolor{ao(english)}{\emph{// Subroutine B: search for ``cost-effective'' candidate trails}}}
\State Find $P\in S$ such that $H\Delta P$ has no vertex of degree greater than $2$ maximizing the cardinality of $H\Delta P$ \label{alg.step.find.long.P}
\If{Step~\ref{alg.step.find.long.P} finds a $P$ and $|H\Delta P|\geq |H|+\sqrt{\log n}$}
\State Set $H=H\Delta P$
\EndIf
\If{$H$ was not updated in lines 5-12}
\State Set $can\_grow = False$.
\EndIf
\EndWhile
\State \Return $H$
\end{algorithmic}
\end{algorithm}

First, we generalize Lemma~\ref{lemma:difference-graph} and prove that the difference graph $H \Delta H^*$ admits a trail decomposition such that each trail alternates between red (planted) and blue (unplanted) edges when it goes in and out of a vertex in $V(H) \cap V(H^*).$ 
\begin{lemma}\label{lmm:decomposition_new}
Let $H^{\ast}$ be a planted $2$-factor on $m$ vertices, and let $H$ be a graph where every vertex has degree at most $2$ and there are $2m$ vertices of degree $1$ in $H$. Then
\begin{itemize}
    \item The difference graph $H^{\ast} \Delta H$ is a graph where every vertex has degree at most $4$. Every degree-$4$ vertex has two incident red (planted) edges and two incident blue (unplanted) edges, and every degree-$3$ vertex has two incident red (planted) edges and one incident blue (unplanted) edge.
    \item 
    Furthermore, $H^{\ast} \Delta H$ can be decomposed into a collection of edge-disjoint  trails such that each trail must alternate between  red and blue edges whenever it goes in and out of a vertex in $V(H^*)\cap V(H)$, and there are exactly $m$ open trails whose endpoints must have degree 1 in $H$. 
    
\end{itemize}
\end{lemma}
\begin{proof}
The first part of the statement readily follows from the fact that every vertex in $H^*$ has degree $2$ and every vertex in $H$ has degree at most $2.$ Moreover, a vertex  has degree $1$ or $3$ in $H^*\Delta H$ if and only if it has degree $1$ in $H$. Thus, there are $2m$ vertices of degree $1$ or $3$ in $H^* \Delta H$. 

To prove the second part of the statement, we transform the difference graph $H\Delta H^*$ by splitting the vertices of degree at least $3$ as follows. If vertex $v$ has degree at least $3$, we split it into two copies, and split the incident edges to $v$ such that one red edge and one blue edge are incident to one copy, while the rest are incident to the other copy. After the transformation, every vertex in $H \Delta H^*$ has degree at most $2$, and there are $2m$ vertices of degree $1$, which exactly correspond to the $2m$ vertices of degree $1$ in $H.$
Thus, the difference graph can be decomposed into a collection of edge-disjoint trails, including exactly $m$ open trails whose endpoints must have degree $1$ in $H.$ 
Moreover, whenever a trail goes in and out of a vertex in $V(H^*) \cap V(H)$, it must alternate between red and blue edges. To get the trail decomposition for the original $H\Delta H^*$, we simply merge the copies of vertices.  

\end{proof}

Building upon the above trail decomposition, we show that any $H$ consisting of cycles and $o(n)$ paths with a total of $\delta n- o(n)$ edges achieves almost exact recovery of $H^*$.

\begin{lemma}\label{lmm:alg.ae.recovery}
Let $\delta\in (0,1]$, $0<\lambda<\frac{1}{(\sqrt{2\delta}+\sqrt{1-\delta})^2}$,  and $G\sim \calG(n,\lambda,\delta)$ conditional on $H^*=h$. Let $H$ be a (possibly random) subgraph of $G$ such that every vertex in $H$ has degree at most $2$, $\mathbb{E}[|H|]=\delta n-o(n)$, and the expected number of degree-$1$ vertices in $H$ is $o(n)$. Then $\mathbb{E}[|H^* \Delta H| ]=o(n)$. 
\end{lemma}

\begin{proof}

The proof is similar to \prettyref{thm:possibility_almost_exact_partial_2_factor}, in that it first decomposes $H\Delta H^*$ into  trails (closed or open), and then argues that not many long trails can have more than $(1/2-\epsilon)$ fraction of their edges unplanted. That allows us to bound the number of unplanted edges in $H\Delta H^*$, which in turn yields a bound on the total number of edges in $H\Delta H^*$, since $H\Delta H^*$ is almost balanced under the assumptions on $H$.

Formally, we start by applying Lemma~\ref{lmm:decomposition_new} to decompose $H\Delta H^*$ into a collection of trails
$\{P_i\}_{1\leq i\leq m}$ with $a_i$ planted and $b_i$ unplanted edges respectively, such that the number of open trails is exactly equal to half of the number of degree-1 vertices in $H$, and each trail alternates between red and blue edges whenever it goes in and out of a vertex in $V(H) \cap V(H^*).$

For a trail $P$ with $a$ planted edges and $b$ unplanted edges, define the excess unplanted edges of $P$ to be $\ex(P)=b-(1/2-\epsilon)(a+b)$, where $\epsilon$ is a constant to be specified later. On one hand, we have
\[
|H\backslash H^*| = \sum_{i=1}^m b_i 
= \sum_{i=1}^m \ex(P_i) + \left(\frac{1}{2}-\epsilon\right)|H\Delta H^*|.
\]
On the other hand,
\[
|H\Delta H^*| = |H\backslash H^*| + |H^*\backslash H| = 2|H\backslash H^*| + \left(
 |H^*| - |H|\right),
\]
where the last equality is because $|H^*\backslash H|-|H\backslash H^*| = |H^*| - |H|$. Combining the two displays above gives
\[\epsilon |H \Delta H^{\ast}| = \sum_{i=1}^m \ex(P_i) - |H \backslash H^{\ast}| + \frac{1}{2}|H \Delta H^{\ast}| = \sum_{i=1}^m \ex(P_i) + \frac{1}{2} \left(|H^{\ast}| - |H| \right) \]
so that
\[
|H\Delta H^*| = \frac{1}{\epsilon}\left(\sum_{i=1}^m \ex(P_i) + \frac{1}{2}\left(
 |H^*| - |H|\right)\right).
\]
By assumption, $\mathbb{E}[|H^*|-|H| ]=\delta n - \mathbb{E}[|H| ]=o(n)$. We only need to bound the expected number of excess unplanted edges; namely, $\mathbb{E}[\sum_{i\leq m} \ex(P_i) ]$. Let $I\subset [m]$ index those trails $P_i$ that are closed; \ie, those which are circuits. Then $\mathbb{E}[\sum_{i\in I}\ex(P_i) ]$ can be bounded in the same way as in the proof of \prettyref{thm:possibility_almost_exact_partial_2_factor}. To recap, we have
$$
\sum_{i\in I} \ex(P_i) \le \sum_{C\subset G} \ex(C)\indc{\ex(C)>0} \triangleq \Gamma,
$$
where we sum over all $(a,b)$-circuits $C$ in $G$ with $a\ge 0$ and $b \ge 1. $ Moreover, in the proof of \prettyref{thm:possibility_almost_exact_partial_2_factor}, we have shown that 
$\expect{\Gamma } \le C 
$
for some constant $C$ independent of $n.$
Therefore, 
$$
\expect{\sum_{i\in I}\ex(P_i) }
\le \expect{\Gamma } \le C. 
$$

The same argument does not work for those $P_i$ that are open, since open trails start and end at two different vertices, so that $G$ may contain $\Theta(n)$ open trails with large excess. For a tighter analysis, we need to utilize the assumption that there are not many degree-$1$ vertices in $H$. 

For each $i \in I^c$, if $P_i$ both starts and ends with an all-red segment, then we remove the initial red segment; if $P_i$ both starts and ends with an all-blue segment, then we split $P_i$ into two trails, where the last segment becomes a new trail by itself. Let $P'_1$, ..., $P'_{M}$ denote the resulting collection of trails for some $M$. 
Crucially, each $P'_i$ (reversing the traversal direction if necessary) must start with unplanted edges and end with planted edges if there are any; thus, it is a valid $(a,b)$-trail for some $a\ge 0$ and $b \ge 1$. Moreover, since each $P'_i$ contains at least one distinct degree-1 vertex in $H$, $M$ is upper bounded by the number of degree-$1$ vertices in $H$ and hence $\expect{M }=o(n). $
Therefore
\begin{align*}
\sum_{i\in I^c} \ex(P_i) \leq & \sum_{i=1}^M \ex(P'_i)
\leq \sum_{i=1}^M \sum_{\ell=0}^\infty \indc{\ex(P'_i)>\ell}
=  \sum_{\ell = 0}^\infty \sum_{i=1}^M \indc{ \ex(P'_i)>\ell}
\leq  \sum_{\ell =0}^\infty \min\left\{M,N_\ell\right\},
\end{align*} 
where $N_\ell \triangleq \sum_{P \subset G} 
\indc{\ex(P) >\ell}$ 
and the sum is over all possible $(a,b)$-trails $P$ in $G$ with $a \ge 0$ and $b \ge 1$. 
Taking the expected value on both sides yields that
\begin{align}
\mathbb{E}\left[\sum_{i\in I^c}\ex(P_i)  \right]\leq \sum_{\ell = 0}^\infty \mathbb{E}\left[\min\left\{M,N_\ell\right\} \right] \leq \sum_{\ell = 0}^\infty \min\left\{\expect{M },\expect{N_\ell } \right\}, \label{eq: non.cycle.excess}
\end{align}
where the second inequality follows from Jensen's inequality as $\min\{x,y\}$ is a concave function of $(x,y).$
By Lemma \ref{validWalkCountLem}, for any two given vertices $v,v'$, the expected number of $(a,b)$-trails from $v$ to $v'$ is at most $c_{a,b}/(\delta n)$ if $a \ge 1, b \ge 1$ 
and $(1-\delta)^{b-1} \lambda^b/n$ if $a=0, b \ge 1.$ Therefore,
\begin{align*}
\mathbb{E}\left[N_\ell \right] &\leq  n^2 \sum_{a, b \ge 1} \frac{c_{a,b}}{\delta n} \indc{ b-(1/2-\epsilon)(a+b)>\ell}  + n^2 \sum_{b\ge 1}
 \frac{(1-\delta)^{b-1} \lambda^b}{n} \indc{(1/2+\epsilon)b>\ell}
\\
&\le  \frac{n}{\delta}\sum_{a,b \ge 1} c_{a,b} 
\indc{ b-(1/2-\epsilon)(a+b)>\ell} + 
\frac{n}{(1-\delta)} 
\frac{[(1-\delta)\lambda]^{\ell}}{1-(1-\delta) \lambda},
\end{align*}
where the second inequality holds because $(1-\delta)\lambda<1.$

As in the proof of \prettyref{thm:possibility_almost_exact_partial_2_factor}, we choose $0<\epsilon<1/2$ and $x<1,y<\frac{1}{\lambda(1-\delta)}$ such that $\frac{2x}{1-x}\cdot\frac{\delta\lambda y}{1-(1-\delta)\lambda y}<1$ and $x^{1+2\epsilon}y^{1-2\epsilon}=1$. Such $\epsilon,x,y$ always exist when $0<\lambda<\frac{1}{(\sqrt{2\delta}+\sqrt{1-\delta})^2}$. Therefore,
\begin{align*}
 \sum_{a,b \ge 1} c_{a,b} 
\indc{ b-(1/2-\epsilon)(a+b)>\ell}
= & \sum_{a,b \ge 1} c_{a,b}  \indc{ b-(1/2-\epsilon)(a+b)>\ell} \left(x^{1/2+\epsilon}y^{1/2-\epsilon}\right)^{a+b}  \\
= & \sum_{a,b \ge 1} c_{a,b}  \indc{ b-(1/2-\epsilon)(a+b)>\ell} x^a y^b\left(\frac{x}{y}\right)^{b-(1/2-\epsilon)(a+b)}  \\
\le  & \left(\frac{x}{y}\right)^{\ell}  \sum_{a, b \ge 1} c_{a,b} x^a y^b\\
\le & \left(\frac{x}{y}\right)^{\ell} \frac{1}{1-\frac{2x}{1-x}\cdot\frac{\delta\lambda y}{1-(1-\delta)\lambda y}},
\end{align*}
where the first inequality holds because $x/y<1$; and the second inequality follows from the generating function of $c_{a,b}$ given by~\eqref{eq:generating_function}.


Combining the last two displayed inequalities yields that 
$\expect{N_\ell}\leq c_1 c_2^\ell n$ for some constants $c_1>0$ and $0<c_2<1$ that are independent of $n, \ell$. Plugging it back into \eqref{eq: non.cycle.excess} and defining 
$\ell_0 \triangleq 
\left\lceil \frac{\log \left(\frac{c_1 n}{\expect{M }}\right)}{\log (1/c_2)}\right\rceil$ yields that 
\begin{align*}
\mathbb{E}\left[\sum_{i\in I^c}\ex(P_i) \right]
& \le \sum_{\ell = 0}^\infty \min\left\{\expect{M },c_1 c_2^\ell n\right\} \\
& \leq  \expect{M } \ell_0
 + \sum_{\ell \ge \ell_0} c_1 c_2^\ell n \\
 & \leq \expect{M } \ell_0 + \frac{\expect{M}}{1-c_2}
 =o(n),
\end{align*}
where the last  inequality holds because $c_1 c_2^{\ell_0} n \le \expect{M }$ and the last equality holds because $\expect{M }=o(n)$ and $\ell_0=O(\log(n/\mathbb{E}[M]))=o(n/\mathbb{E}[M])$. 

Finally, we conclude that
\begin{align*}
\expect{|H\Delta H^*|  } & =\frac{1}{\epsilon}\left(\mathbb{E}\left[\sum_{i\in I}\ex( P_i) \right] + \mathbb{E}\left[\sum_{i\notin I}\ex( P_i)\right] + \frac{1}{2} \expect{|H^*|-|H| } \right) \\
& = O(1) + o(n)+o(n)=o(n).
\end{align*}
\end{proof}

\begin{lemma}
\label{lmm:alg.Hgrows}
Let $G$ be a graph that contains a $2$-factor $H^*$ on $\delta n$ vertices. 
 The   $H$ output by Algorithm \ref{alg:search} on input $G$ 
will have at least $\delta n-O(n/\sqrt{\log n})$ edges and at most $2n/\sqrt{\log n}$ vertices of degree $1$. 
\end{lemma}


\begin{proof}

To prove that Algorithm~\ref{alg:search} outputs a large enough subgraph, we argue that until $H$ contains $\delta n - O(n/\sqrt{\log n})$ edges, in each iteration of the main loop of the algorithm, we can find some trail (closed or open) to XOR  onto $H$ to make it bigger, through either subroutine A or B. To show that the output does not contain too many degree-$1$ vertices, we will argue that only subroutine B can increase the number of degree-$1$ vertices in $H$. The increase is by at most 2, and must also increase $H$ by at least $\sqrt{\log n}$. 

In more detail, assume we are at a certain iteration of the while loop of the algorithm with an iterate of $H$. By the algorithm's construction, $H$ does not contain any vertices of degree greater than $2$. In other words, $H$ is a vertex-disjoint union of cycles and paths. On a high level, we prove the lemma by arguing that there exists a subgraph $P$ of $H\Delta H^*$ that makes a ``good'' candidate for the XOR operation, updating $H$ to $H\Delta P$. In particular, we will invoke Lemma~\ref{lmm:decomposition_new} to decompose $H\Delta H^*$ into a collection $\mathcal{C}$ of trails. Next, we show that some $P\in \mathcal{C}$, or slight variation of it, qualifies as a candidate trail to update $H$.


To begin with, we state some simple facts on $H\Delta P$ that will be handy in our later arguments. 

\begin{fact}\label{fact_trail}
\begin{enumerate}[(a)]
\item For each $P\in \mathcal{C}$, the XOR operation $H\Delta P$ simply adds the planted edges of $P$ to $H$, and removes the unplanted edges of $P$ from $H$.
\item For $P\in \mathcal{C}$, all degree-$1$ vertices in $H\Delta P$ are also of degree $1$ in $H$, and $H\Delta P$ does not contain any vertices of degree greater than $2$. 
\end{enumerate}
\end{fact}
\begin{proof}[Proof of Fact~\ref{fact_trail}]
Fact (a) holds simply because $P$ is a subgraph of $H\Delta H^*$, thus the planted edges in $P$ are in $H^\ast\backslash H$ and the unplanted edges are in $H\backslash H^\ast$. 

To prove Fact (b), we first argue that vertices of even degree in $H$ remain even degree in $H\Delta P$. 
That is because if the degree of a vertex $v$ is even in $H$ and odd in $H\Delta P$, then $v$ must have an odd degree in $P$, which is only possible if $P$ is an open trail and $v$ is one of its endpoints. However, by Lemma~\ref{lmm:decomposition_new}, all endpoints of open trails must be degree-$1$ vertices in $H$, forming a contradiction. 

Since $H$ only contains vertices of degree $1$ or $2$, $H\Delta P$ does not contain any new degree-$1$ vertices.  To prove that it cannot create any vertices of degree greater than $2$, we only need to show that degree-$1$ vertices in $H$ cannot have degree $3$ in $H\Delta P$, and degree-$2$ vertices in $H$ cannot have degree $4$ in $H\Delta P$. For either case to occur, the vertex $v$ must be of degree $3$ or $4$ in $H\Delta H^*$, and the XOR operation must be adding the two planted edges incident to $v$ to $H$. However, Lemma~\ref{lmm:decomposition_new} ensures that whenever $P\in \mathcal{C}$ goes in and out of a vertex of degree $3$ or $4$ in $H\Delta H^*$, it goes through a planted and an unplanted edge. Therefore the XOR operation can never only add $2$ planted edges without removing any unplanted edges. Thus no vertex in $H\Delta P$ can be of degree more than $2$.
\end{proof}

By Fact~\ref{fact_trail}(b), all $P\in \mathcal{C}$ are automatically ``cost-free,'' in the sense that they can be XOR'ed onto $H$ without increasing the number of degree-$1$ vertices or creating any vertices of degree greater than $2$. However, this does not guarantee that $H$ is updated in subroutine A, since Line~\ref{alg.step.find.short.P} is only run if the subroutine finds a ``cost-free'' trail that is a) of length less than $\log n$, and b) strictly increases $|H|$. We discuss the following two cases. Either there is some member of $\mathcal{C}$ that satisfies both constraints, in which case $H$ is guaranteed to be updated in subroutine A; or no $P\in \mathcal{C}$ satisfies both constraints, in which case we will argue that $H$ must be updated in subroutine B if subroutine A leaves $H$ unchanged. 



{\bf Case 1:} There exists some $P\in \mathcal{C}$ such that $|P|< \log n$ and  $|H\Delta P| >|H|$. We claim that in this case, subroutine A successfully finds at least one ``cost-free'' candidate trail $P$. 

Since $H$ only contains edges in $G$ by construction, and $P$ is from the decomposition of $H\Delta H^*$, we have that $P$ belongs to the set $S$ of all trails in $G$ of length smaller than $\log n$. In subroutine A, the algorithm searches through all members of $S$ in a for-loop. At some point, it will encounter $P\in S$. If $H$ has already been updated in the for-loop before it reaches $P$, then we know that $H$ is updated at least once in the for-loop; otherwise $P$ is XOR'ed onto $H$. Either way, we have that at the end of the for loop, subroutine A successfully finds at least one ``cost-free'' candidate trail, and $H$ gets updated. Therefore $|H|$ strictly increases.

{\bf Case 2:} There exists no $P\in \mathcal{C}$ such that $|P|<\log n$ and $|H\Delta P|> |H|$. 
In this case, we claim that if subroutine A leaves $H$ unchanged, subroutine B is guaranteed to find a ``cost-effective'' trail. The high-level reasoning is as follows: in Case 2, all ``short'' trails of length smaller than $\log n$ in $\mathcal{C}$ do not increase $|H|$ when XOR'ed onto $H$, meaning that they all contain at least as many unplanted as planted edges. However, before $H$ grows to contain $\delta n-O(n/\sqrt{\log n})$ edges, there are still a large number of planted edges not picked up by $H$. Thus, the collection $\mathcal{C}$ (which comes from the decomposition of $H\Delta H^*$) must contain some long trail with a lot more planted than unplanted edges. 
As a result, as subroutine B searches through trails of length up to $\log n$, there exists some segment of a long trail in $\mathcal{C}$ that qualifies as ``cost-effective,'' ensuring that $H$ is updated in subroutine B. Now, we present the rigorous proof below. 

Since
\[
(\#\text{ of planted edges in }H\Delta H^*) - (\#\text{ of unplanted edges in }H\Delta H^*) = |H^*\backslash H| - |H\backslash H^*| = |H^*| - |H|,
\]
the overall collection $\mathcal{C}$ contains $|H^*| - |H|$ more planted edges than unplanted edges. Also, recall that the ``short'' trails contain at least as many unplanted edges as planted ones. Therefore, those $P\in \mathcal{C}$ of length at least $\log n$ must have at least $|H^*| - |H|$ more planted edges than unplanted edges. For each $P\in \mathcal{C}$ of length at least $\log n$, we divide it into segments of length $\log n/2.$ 
If there are edges left over, we combine it with the last segment. This way, all long trails in $\mathcal{C}$ are chopped into trails whose lengths are in $[\lfloor \log n/2\rfloor,\log n)$. As long as $|H|<\delta n-cn/\sqrt{\log n}$ for a large enough constant $c$, we claim that at least one of these chopped trails contains at least $2\sqrt{\log n}$ more planted than unplanted edges. Otherwise, the total number of chopped trails must be at least $(|H^*| - |H|)/(2\sqrt{\log n}) > cn/ (2\log  n)$, amounting to a total of at least $\frac{cn}{ 2 \log n}\cdot\frac{\log n}{2} = cn/4$ edges. For large enough $c$ ($c=9$ suffices) this is impossible since the number of edges in $\mathcal{C}$ is bounded by $|H\Delta H^*|  \leq |H|+|H^*|\leq 2\delta n$. 

Now we have found a trail $P'$ that contains at least $2\sqrt{\log n}$ more planted than unplanted edges, which means that it can be XOR'ed onto $H$ so that $|H\Delta P'|\geq |H| + 2\sqrt{\log n}$. For Line 12 of \prettyref{alg:search} to be run, we also need to check that $H\Delta P'$ does not contain any vertices of degree greater than $2$. Note that since $P'$ is only a segment of some $P\in\mathcal{C}$, Fact~\ref{fact_trail}(b) does not apply directly to $H\Delta P'$. However, Lemma~\ref{lmm:decomposition_new} entails that whenever $P$ goes in and out of a vertex $v\in V(H)\cap V(H^*)$, it must go through a planted edge $e$ and an unplanted edge $e'$. Therefore, the only way for $v$ to have degree $3$ or more in $H\Delta P'$ is if $P$ is chopped at $v$, and $P'$ contains the planted edge $e$. In other words, $P'$ starts or ends at $v$ through the planted edge $e$.
In this case, we simply remove the planted edge $e$ from $P'$. Doing the same for both endpoints of $P'$ results in a trail $\widetilde{P}$, which is $P'$ with at most two planted edges removed. For each vertex in $\widetilde{P}$, its degree in $H\Delta \widetilde{P}$ is bounded by its degree in $H\Delta P'$, which is at most $2$. Moreover, since $P'$ contains at least $2\sqrt{\log n}$ more planted than unplanted edges, we have that $\widetilde{P}$ contains at least $2\sqrt{\log n}-2\geq \sqrt{\log n}$ more planted than unplanted edges. Thus $|H\Delta \widetilde{P}|\geq |H| + \sqrt{\log n}$.

To summarize, in Case 2, there exists some trail that when XOR'ed onto $H$, does not create any vertex of degree greater than $2$, and increases $|H|$ by at least $\sqrt{\log n}$. Therefore, if $H$ is not updated in subroutine A, it must be updated in subroutine B. That is, Line~\ref{alg.step.find.long.P} of Algorithm~\ref{alg:search} is run, the update increases $|H|$ by at least $\sqrt{\log n}$, and increases the number of degree-$1$ vertices in $H$ by at most 2. That is because updating $H$ to be $H\Delta P$ can only create a new degree $1$ vertex if that vertex is an endpoint of $P$. 


Combining Cases 1 and 2, we have that until $H$ contains $\delta n- O(n/\sqrt{\log n})$ edges, each while loop iteration updates $H$ to be bigger. Thus the algorithm outputs an $H$ with at least $\delta n- cn/\sqrt{\log n}$ edges for some large enough constant $c$. Moreover, for the update to increase the number of degree-$1$ vertices by $1$ or $2$, the update must be through subroutine B, which simultaneously increases $H$ by at least $\sqrt{\log n}$. Given that the output $H$ has degree at most $2$, we have $|H|\leq n$ and thus it contains at most $2n/\sqrt{\log n}$ degree-$1$ vertices.

\end{proof}

\begin{lemma}\label{lmm:poly.time}
Let $\delta\in (0,1]$, $\lambda>0$ be a constant, and $G\sim \calG(n,\lambda,\delta)$ conditional on $H^*$. Then Algorithm \ref{alg:search} on input $G$ runs in time $O\left(n^{3 + \log(2 +\lambda)} \right)$ with high probability. 
\end{lemma}
\begin{proof}
The runtime of Algorithm \ref{alg:search} is the sum of the runtime required to find the set $S$ comprising the trails in $G$ of length less than $\log n$, along with the runtime required to execute the while loop.

For convenience, we use the adjacency list representation of $G$, where we record the list of neighbors of each vertex.
In order to find the set of trails in $G$ of length less than $\log n$ which begin at a particular vertex $v$, we employ a breadth-first traversal. In expectation, at most $\sum_{k=1}^{\log n} (2+\lambda)^{k}$ trails beginning at $v$ will be discovered, where $(2+\lambda)^{k}$ is an upper bound on the expected number of trails at depth exactly $k$ beginning at $v$. Therefore, Markov's inequality implies that with high probability, at most
$n^{1+\epsilon}\sum_{k=1}^{\log n} (2+\lambda)^{k} \leq n^{1+\epsilon} \log n \cdot (2+\lambda)^{\log n}$ trails will be discovered in total for any $\epsilon > 0$. Each trail requires visiting up to $\log n$ edges, and we can keep track of already-visited edges by modifying the adjacency list representation (delete visited edges and then add them back once a trail is complete). Therefore,  finding these trails requires $O\left(n^{1+\epsilon} \log^2 n  \cdot (2+\lambda)^{\log n} \right)$ work.

Since Algorithm \ref{alg:search} increases the number of edges in $H$ whenever it updates $H$ and no vertex in $H$ ever has degree greater than $2$, there are at most $n$ updates. Each update requires searching through the set $S$ and performing $O(\log n)$ work for each $P \in S$. It follows that the total runtime of the while loop is
\[O\left(n \cdot |S| \cdot \log n \right) =  O\left(n^{2+\epsilon} \log^2 n \cdot (2+\lambda)^{\log n} \right) = O\left(n^{3 + \log(2 +\lambda)} \right).\]
Since the while loop dominates the runtime, we conclude that the overall runtime of Algorithm \ref{alg:search} is $O\left(n^{3 + \log(2 +\lambda)} \right)$ with high probability. 

\end{proof}

\section{Discussion and future work}\label{sec:discussion}
In this paper, we have characterized a sharp threshold of almost exact recovery for the planted cycles model $\mathcal{G}(n, \lambda, \delta)$. A precursor paper by the authors \cite{gaudio2025all} examined the special case $\delta = 1$—where the planted subgraph spans the entire vertex set—and considered more generally planted $k$-factors for $k \in \mathbb{N}$.  Parameterizing the edge probability in the background graph as $p_n = \lambda_n/n$, \cite{gaudio2025all} discovers a curious ``All-Something-Nothing'' transition in recovery as $\lambda_n$ increases: when $\lambda_n = o(1)$, exact recovery is possible; when $\Omega(1) = \lambda_n \leq 1/k$, almost exact recovery remains feasible; when $1/k <\lambda_n \le O(1)$, only partial recovery is achievable; and once $\lambda_n = \omega(1)$, even partial recovery becomes impossible. This multi-phase recovery landscape contrasts sharply with the ``All-or-Nothing'' transition established by an earlier work \cite{Mossel2023}, for the recovery of a planted graph $H^*$ within a background graph $G \sim \mathcal{G}(n,p_n)$, whenever $H^*$ satisfies certain properties related to size, density, and balancedness. In those cases, the recovery transition between almost exact recovery (``all'') and impossibility of partial recovery (``nothing'') occurs at a threshold that depends on the first-moment thresholds of subgraphs $J \subset H^*$. ``All-or-Nothing'' transitions are also known to occur in a variety of other high-dimensional inference problems, including sparse linear regression~\cite{reeves2021all}, sparse tensor PCA~\cite{niles2020all}, group testing~\cite{truong2020all,coja2022statistical}, and graph matching~\cite{wu2022settling}.

Our results taken together identify the ``All'' regime (Theorems \ref{thm:almost-exact} and \ref{thm:exact}) and partially identify the ``Nothing'' regime (Theorem \ref{thm:nothing}). Yet a compelling open question remains: is there a ``Something'' regime in between? That is, for constant $\lambda>(\sqrt{2\delta} + \sqrt{1 - \delta})^{-2}$, can we achieve partial recovery in the sense that   $\Expect[\risk(\Mplanted,\hat H)] \le 1-\Omega(1)$? 
In the special case of $\delta = 1$, one can partially recover the planted 2-factor with no errors by simply collecting edges incident to degree-2 vertices. However, for $\delta < 1$, this simple structural cue disappears, and recovery becomes subtler. A na\"ive alternative of randomly sampling $\delta n$ edges does recover a constant fraction of planted edges, but also mistakenly includes many unplanted ones, resulting in a high overall error and failing to achieve partial recovery. 

We list some additional open problems below.
\begin{itemize}
\item If the $\mathcal{G}(n, \lambda, \delta)$ model admits a ``something'' phase, what is the optimal recovery rate in the partial recovery regime, and can it be achieved by an efficient algorithm? 
\item 

Consider a general $\mathcal{G}(n, \lambda, \delta, k)$ model in which a partial $k$-factor is planted on a subset of $\delta n$ vertices in the background graph, so that the case $k=2$ reduces to the planted cycles model $\mathcal{G}(n, \lambda, \delta)$. When $k=1$ and $\delta \in (0,1)$, this corresponds to the planted partial matching problem, for which there is no sharp phase transition for almost-exact recovery: if $\lambda = o(1)$, recovery is possible since there are only $o(n)$ unplanted edges, whereas if $\lambda = \Omega(1)$, recovery becomes impossible because a partial matching of size $\Theta(n)$ formed by unplanted edges will exist, and there is no way to tell apart the edges in this partial matching from the edges in the planted one. This raises the question: what are the recovery thresholds in this generalized model for $k \ge 3$? When $k \ge 3$, the transition from almost-exact recovery to impossibility may be driven by the emergence of $k$-factors which differ significantly from the planted one, rather than $k$-factors that differ slightly.

\item Similarly, what can we say about planted subgraph recovery thresholds for graphs other than $k$-factors? 
\item What can be said for weighted graphs, where planted edges have weights drawn from a distribution $\mathcal{P}$ while unplanted edges have weights drawn from a distribution $\mathcal{Q}$? For the special case of a weighted planted matching where $\mathcal{P}$ and $\mathcal{Q}$ follow exponential distributions, the recovery rate has been precisely characterized \cite{moharrami2021planted}.
\end{itemize}

\appendices

\section{History-dependent branching process}

In order to characterize the sizes of the trees constructed by~\prettyref{alg:tree-construction2}, we compare the trees to branching processes.
At a high level, the probability that a given tree reaches a prescribed depth can be related to the survival probability of the branching process. We need the following auxiliary result about the survival of a supercritical branching process (used in the proof of Lemma~\prettyref{lemma:enough-trees2}). 

\begin{lemma}\label{lmm:branching}
Suppose a branching process has offspring distribution with expected value $\mu$ and variance $\sigma^2$ for some $\mu>1$, we have
\begin{equation}
\mathbb{P}\{\text{Branching process survives}\} \geq \frac{\mu^2-\mu}{\mu^2-\mu+\sigma^2}. \label{eq:branchin-brocess-survival}
\end{equation}
A similar bound holds for an inhomogeneous, possibly history-dependent branching process. Specifically, let $X_1, X_2, \ldots $ be the number of children born from vertices in breadth-first order, and let $\mathcal{F}_1=\sigma(X_1), \mathcal{F}_2=\sigma(X_1, X_2), \dots$ be the natural filtration. Suppose that $\mathbb{E}\left[X_i \mid \mathcal{F}_{i-1} \right] \geq \mu > 1$ and $\text{Var}(X_i \mid \mathcal{F}_{i-1}) \leq \sigma^2$ uniformly for all $i \in \{1, 2, \dots\}$.  Then
\[\mathbb{P}\{\text{Branching process survives}\} \geq \frac{\mu^2-\mu}{\mu^2-\mu+\sigma^2 + \frac{1}{4}}.\]

\end{lemma}

The proof of Lemma \ref{lmm:branching} requires the following fact about distributions related in stochastic order.
\begin{proposition}\label{proposition:distribution-shift}
Let $P$ be a finite discrete distribution with support $S \subseteq \{0, 1, 2, \dots, s\}$. Let $\mu$ and $\sigma^2$ respectively denote the mean and variance of $P$. Then for any $\mu' \in [0,\mu)$, there exists a distribution $Q$ with support $S' \subseteq \{0, 1, 2, \dots, s\}$, mean $\mu'$, and variance $\sigma'^2$ such that $Q \preceq P$ and $\sigma'^2 \leq \sigma^2 + \frac{1}{4}$. 
\end{proposition}
\begin{proof}
We construct the distribution $Q$ from $P$ by iteratively ``shifting'' mass from larger elements of $S$ to smaller ones. Observe that any distribution shifted from $P$ in this way is automatically a stochastic lower bound on $P$. Shifting always reduces the mean, and may increase or decrease the variance.

Without loss of generality, $P$ places nonzero mass on $s$ (otherwise, we could redefine $S$).
Suppose we shift the probability mass 
 $x$ from $s$ to $s-1$,
 where $x \in [0, P(s)]$,
 and let $Q'$ denote the resulting probability distribution. 
 Next, we bound the change of the variance.
To this end, we define a coupling $(X,X')$, where $(X,X')=(i,i)$ with probability $P(i)$ for all $i \le s-1$, 
$(X,X')=(s,s)$ with probability $P(i)-x$, and $(X,X')=(s,s-1)$ with probability $x$.
It follows that $X$ is distributed as $P$ and $X'$ is distributed as $Q'.$ 
Let $(Y, Y')$ be an independent copy of $(X, X').$
It follows that 
\begin{align*}
\text{Var}(X')
-\text{Var}(X)
& =\frac{1}{2}\expect{\left(X' - Y' \right)^2 - \left(X - Y \right)^2}\\
& = \frac{1}{2}\expect{
\left(\left(X' - Y' \right)^2 - \left(X - Y \right)^2\right) \indc{X=s, X'=s-1, Y =Y' \text{ or } X=X', Y =s, Y'=s-1 }} \\
& = \expect{
\left(\left(s-1 - X' \right)^2 - \left(s - X\right)^2\right) \indc{X =X'}} \cdot x \\
& \le  \prob{X=X'=s} \cdot x  = (P(i)-x) \cdot x ,
\end{align*}
where the first equality holds by the definition of variance; the second equality holds because $X\neq X'$
if and only if $X=s$ and $X'=s-1$, and the same holds for $Y,Y'$;
the third equality holds because $(X,X')$ and $(Y, Y')$ are independent copies, and $Y=s, Y'=s-1$ with probability
$x$; the last inequality holds because $(s-1-X')^2 
\le (s-X)^2$ when $X=X' \le s-1$ and 
$(s-1-X')^2 - (s-X)^2=1$
when $X=X'=s.$ Therefore, if $x=P(i),$
then the variance of $Q'$ is always no larger than that of $P$;
if $x<P(i)$, then the variance increases at most by $x(1-x)\le 1/4.$

Therefore, we can iteratively shift all the probability mass $x=P(s)$ from $s$ to $s-1$, where $s$ decreases by 1 in each iteration, until the very last iteration, where we only shift $ x\le P(s)$  so that the mean decreases to exactly $\mu'.$ By the above argument, the variance keeps non-increasing until the very last iteration, in which it may increase by at most $1/4.$ Therefore, we arrive at the final distribution $Q \preceq P$ with mean $\mu'$ and variance at most $\sigma^2+1/4.$


\end{proof}

\begin{proof}[Proof of Lemma \ref{lmm:branching}]
We first consider the homogeneous case. In that case, the proof mostly follows the derivations in~\cite[Chapter 2.1]{durrett2007random}.
Let $Z_m$ denote the number of vertices in generation $m$. Given $Z_{m-1}$, the conditional first and second moments of $Z_m$ satisfy
\begin{align*}
    \mathbb{E}[Z_m | Z_{m-1}] = & \mu Z_{m-1},\\
    \mathbb{E}[Z_m^2 | Z_{m-1}] = & \mu^2 Z^2_{m-1} + Z_{m-1}\sigma^2.
\end{align*}
Taking expected values on both sides and iterating and noting $Z_0=1$, we have
\[
\mathbb{E}[Z_m] = \mu^m,
\]
and
\begin{align}
\mathbb{E}[Z_m^2]  &= \mu^{2m} + \sigma^2 \sum_{j=m-1}^{2m-2} \mu^j \label{eq:second-moment-recurrence}\\
&\leq  \mu^{2m} + \sigma^2\frac{\mu^{2m-2}}{1-\mu^{-1}}. \nonumber
\end{align}

By the Paley--Zygmund inequality, the probability that the branching process survives to iteration $m$ is
\[
\mathbb{P}\{Z_m\geq 1\}\geq 
\frac{\mathbb{E}[Z_m]^2}{\mathbb{E}[Z_m^2]}
\geq \frac{\mu^{2m}}{\mu^{2m}+\sigma^2\frac{\mu^{2m-2}}{1-\mu^{-1}}}
= \frac{\mu^2-\mu}{\mu^2-\mu+\sigma^2}.
\]
Take $m\rightarrow\infty$ to finish the proof.

For the inhomogeneous, possibly history-dependent case, we will construct an auxiliary branching process, where $\{\overline{Z}_m\}$ denotes the number of vertices in generation $m$ in the auxiliary branching process. Let $\overline{X}_i$ be the number of children of the $i^{\text{th}}$ node in the auxiliary tree. The initial children $\overline{Z}_1$ are sampled from the shifted distribution based on $X_1$, guaranteeing $\overline{Z}_1 = \overline{X}_1 \leq X_1$. The children of the two processes are arbitrarily matched, and their offspring are coupled. For $i \in \{1, 2, \dots, \overline{X}_1\}$, conditioned on $\mathcal{F}_i$, the number of offspring $X_{i+1}$ is coupled to $\overline{X}_{i+1}$, such that $\overline{X}_{i+1}\le X_{i+1}$ and 
$\overline{X}_{i+1}$ has mean $\mu$ and variance at most $\sigma^2+1/4$, as in Proposition \ref{proposition:distribution-shift}. We can continue matching pairs across trees and coupling their offspring. It follows that $\overline{Z}_m \le Z_m$. Moreover, \[\mathbb{E}[\overline{Z}_m] = \mu^m,\]
analogously to the homogeneous case. To compute the conditional second moment $\mathbb{E}\left[\overline{Z}_m^2 \mid \overline{Z}_{m-1} \right]$, we claim that
\[\text{Var}\left(\overline{Z}_m \mid \overline{Z}_{m-1}\right) \leq \left(\sigma^2 + \frac{1}{4}\right) \overline{Z}_{m-1}.\]

To see this, let $N = \overline{Z}_{m-1}$, and let $Y_1, \dots Y_N$ denote the numbers of offspring of the $(m-1)^{\text{th}}$ generation, so that $\overline{Z}_m = \sum_{i=1}^N Y_i$. 
Let $\calF$ denote the history of the auxiliary branching process up to the time when the offspring of the $(N-1)$-th node at the $(m-1)$-th generation is added.
Then 
\begin{align*}
\Var\left(\overline{Z}_m \mid \overline{Z}_{m-1} \right) &= 
\Var\left(\sum_{i=1}^N 
Y_i \right) \\
 & = \expect{ 
 \Var\left(\sum_{i=1}^N 
 Y_i \mid \mathcal{F} \right) 
 }+
 \Var\left(
 \expect{\sum_{i=1}^N 
 Y_i \mid \mathcal{F}} \right)\\
&= \expect{ 
 \Var\left(Y_N \mid \mathcal{F} \right) 
}+
\Var\left(
 \sum_{i=1}^{N-1} Y_i + \mu  \right)\\
 & \le \left(\sigma^2 + \frac{1}{4}\right)
 +\Var\left(
 \sum_{i=1}^{N-1} Y_i\right),
\end{align*}
where the second equality holds by the law of total variance; the third equality holds because $\expect{Y_N \mid \calF}=\mu$; the last inequality holds due to $\var\left(Y_N\mid \calF \right) \le \sigma^2+1/4.$ Thus, we have shown
\[\Var\left(\sum_{i=1}^N 
Y_i \right) \leq \left(\sigma^2 + \frac{1}{4}\right)
 +\Var\left(
 \sum_{i=1}^{N-1} Y_i\right).\]
Recursively applying the above inequality yields that 
\[
\text{Var}\left(\overline{Z}_m \mid \overline{Z}_{m-1} \right) = \text{Var}\left(\sum_{i=1}^N 
Y_i \right)
\le  \left(\sigma^2 + \frac{1}{4}\right) N
=\left(\sigma^2 + \frac{1}{4}\right)\overline{Z}_{m-1}.\]
It follows that 
\[\mathbb{E}\left[\overline{Z}_m^2 | \overline{Z}_{m-1} \right] 
\leq \mu^2 \overline{Z}_{m-1}^2 + \left(\sigma^2 + \frac{1}{4}\right) \overline{Z}_{m-1}. 
\]
Taking expected values and iterating as in the homogeneous case, we obtain
\[\mathbb{E}\left[ \overline{Z}_m^2\right] \leq \mu^{2m} + \left(\sigma^2 + \frac{1}{4}\right)\frac{\mu^{2m-2}}{1-\mu^{-1}}.\]

Finally, we conclude that 
\begin{align*}
\mathbb{P}(Z_m \geq 1) \geq
\mathbb{P}(\overline{Z}_m \geq 1) \geq 
\frac{\mathbb{E}[\overline{Z}_m]^2}{\mathbb{E}[\overline{Z}_m^2]}
&\geq \frac{\mu^{2m}}{\mu^{2m}+\left(\sigma^2 + \frac{1}{4}\right)\frac{\mu^{2m-2}}{1-\mu^{-1}}}\\
&= \frac{\mu^2-\mu}{\mu^2-\mu+\sigma^2 + \frac{1}{4}}.
\end{align*}
\end{proof}

\section{Exact and partial recovery}
We say that $\hat H$ achieves 
\begin{itemize}
\item \emph{exact recovery}, if 
$\risk(\Mplanted,\hat H) =0$ with high probability;
    \item \emph{partial recovery},
if $\Expect[\risk(\Mplanted,\hat H)] \le 1-\Omega(1).$
\end{itemize}

In this section, we show that exact recovery is possible if and only if $\lambda=o(1)$ (Theorem \ref{thm:exact}) and partial recovery is impossible when $\lambda = \omega(1)$ (Theorem \ref{thm:nothing}).

\begin{theorem}\label{thm:exact}
Let $\delta\in (0,1]$. Suppose $G\sim \calG(n,\delta, \lambda)$ is the graph generated from the planted cycles model, and  $H^*$ is the hidden $2$-factor on the set of $\delta n$ vertices. If $\lambda = o(1)$, then with probability $1-o(1)$, $H^\ast$ is the only $2$-factor on $\delta n$ vertices contained in $G$. Conversely, if $\lambda = \Omega(1)$, then with probability $\Omega(1)$, $G$ contains another $2$-factor on $\delta n$ vertices which is different from $H^*$.
\end{theorem}
\begin{proof}
For any $2$-factor $H$ on $\delta n$ vertices, $H^*\Delta H$ can be decomposed into a union of $(a,b)$-circuits such that exactly half of the edges are planted. So, if there exists a $2$-factor $H\ne H^*$ on $\delta n$ vertices then there must exist at least one $(a,b)$-circuit with $a\le b$. The probability that such a circuit exists is at most equal to the expected number of such circuits, and by Lemma \ref{validWalkCountLem}, that is at most
\begin{align*}
&\sum_{b=1}^\infty (1-\delta)^{b-1}\lambda^b+\sum_{a=1}^\infty\sum_{b=a}^\infty c_{a,b}\\
&=\sum_{b=1}^\infty (1-\delta)^{b-1}\lambda^b+\sum_{a=1}^\infty\sum_{b=a}^\infty (\lambda(1-\delta))^b \sum_{k=1}^\infty (2\delta/(1-\delta))^k \binom{a-1}{k-1} \binom{b-1}{k-1}\\
&=o(1)+\sum_{a=1}^\infty\sum_{b=a}^\infty (\lambda(1-\delta))^b \sum_{k=1}^a (2\delta/(1-\delta))^k \binom{a-1}{k-1} \binom{b-1}{k-1}\\
&\le o(1)+\sum_{a=1}^\infty\sum_{b=a}^\infty (\lambda(1-\delta))^b \sum_{k=1}^a (2\delta/(1-\delta))^k 2^{a+b-2}\\
&\le o(1)+\sum_{a=1}^\infty a\max(1,(2\delta/(1-\delta))^a)2^{a-2}\sum_{b=a}^\infty (\lambda(1-\delta))^b 2^{b}\\
&= o(1)+\sum_{a=1}^\infty a\max(1,(2\delta/(1-\delta))^a)2^{a-2} \frac{(\lambda(1-\delta))^a 2^{a}}{1-2\lambda(1-\delta)}\\
&= o(1)+\frac{1}{4-8\lambda(1-\delta)}\sum_{a=1}^\infty a\max(4\lambda(1-\delta),8\lambda\delta)^a\\
&= o(1)+\frac{1}{4-8\lambda(1-\delta)}
\cdot\frac{\max(4\lambda(1-\delta),8\lambda\delta)}{(1-\max(4\lambda(1-\delta),8\lambda\delta))^2}\\
&=o(1).
\end{align*}


Next, suppose $\lambda = \Omega(1)$. We assume that the set of $\delta n$ planted vertices is known. Then we can reduce to the problem of recovering a planted $2$-factor on $m_n = \delta n$ vertices, where the edge probability in the background graph is $\frac{\delta \lambda}{\delta n} = \frac{\delta \lambda}{m_n}$. Since $\lambda = \Omega(1)$, it follows that $\delta \lambda = \Omega(1)$. Thus, the result follows directly from \cite[Theorem 2.1]{gaudio2025all}.
\end{proof}

\begin{theorem}\label{thm:nothing}
Consider the planted $2$-factor model on $\delta n$ vertices, where $\delta \in (0,1)$. When $\lambda = \omega(1)$, partial recovery is impossible.
\end{theorem}
\begin{proof}
As in the proof of Theorem \ref{thm:exact}, we assume that the estimator has access to the vertex set $V(H^*)$. An analogous parameter scaling argument along with \cite[Theorem 2.5]{gaudio2025all} completes the proof.
\end{proof}

\section{Recovery of a single planted cycle}\label{sec:planted-cycle}
We now turn to the problem of recovering a single planted cycle on $\delta n$ vertices, chosen uniformly at random, and planted within $G \sim \mathcal{G}(n,p_n)$. Recall Remark \ref{remark:single-cycle}, which notes that the achievability results (Theorem~\ref{thm:possibility_almost_exact_partial_2_factor} and~\prettyref{thm:alg}) apply to the recovery of a single planted cycle. As for the impossibility direction, we observe that
a random $2$-factor $H^*$ 
has a significant probability of being a cycle, so anything that holds with sufficiently high probability for a random planted $2$-factor carries over to a planted cycle. In particular, we have the following result, which generalizes~\cite[Lemma 1.1]{gaudio2025all}. 

\begin{lemma}\label{lmm:factorCycleLem}
Let $H^*$ be a random $2$-factor on $m$ vertices with $m \ge 3.$  For any $1 \le \ell \le m$, with probability at least 
$1/(1+m2^{-\ell})$,
$H^*$ has at most $\ell$ cycles.  In particular, with probability at least $1/m$, $H^*$ is a cycle.   
\end{lemma}

\begin{proof} 
Every 2-factor can be converted to the cycle decomposition
of a permutation of the vertices. Specifically, we assign each of its cycles a direction and then we have that the
permutation maps each vertex to the next vertex in its cycle. Therefore, every $2$-factor with more than $\ell$ cycles can be converted to a permutation of the vertices in at least $2^{\ell+1}$ different ways by assigning each of its cycles a direction, and no two different $2$-factors can yield the same permutation. Thus, there are at most $m!/2^{\ell+1}$ possible $2$-factors with more than $\ell$ cycles on $m$ vertices. On the other hand, 
there are at least $(m-1)!/2$ possible $2$-factors on $m$ vertices that consists of a single cycle. Therefore, the probability that a random $2$-factor on $m$ vertices has at most $\ell$ cycles  is at least 
$$
\frac{(m-1)!/2}{m!/2^{\ell+1} +(m-1)!/2 }
=\frac{1}{1+m2^{-\ell} }.
$$
When $\ell=1$, this probability is $1/(1+m/2) \ge 1/m$ as $m \ge 2.$
\end{proof}

\begin{corollary}
\label{cor:2_factor_cycle}
Let $P_1$ denote the joint distribution of $(H^*,G)$ generated from the planted $\mathcal{G}(n,\lambda,\delta)$ model. Let $P_2$ denote the distribution of $(H^*,G)$ conditional on $H^*$ being a cycle. Then for any event $E$ defined in terms of $(H^*,G)$, it holds that $P_2[E]\le \delta n \cdot P_1[E]$
\end{corollary}
\begin{proof}
Let $C$ be the event that $H^*$ is a cycle, and note that $P_1$ conditioned on $C$ is $P_2$. Therefore,
\begin{align*}
P_1[E]\ge P_1[E,C]=P_1[E|C]\cdot P_1[C]=P_2[E]\cdot P_1[C] \ge P_2[E]/(\delta n),
\end{align*}
where the last inequality holds by~\prettyref{lmm:factorCycleLem}.
\end{proof}

The impossibility result follows immediately.
\begin{lemma}\label{lemma:impossibility-cycle}
Theorem \ref{thm:impossibility-delta} continues to hold when $H^*$ is a planted cycle on $\delta n$ vertices.
\end{lemma}
\begin{proof}
Let $E = \{\ell(H^*, \hat{H}) < \epsilon'\}$. Then Theorem \ref{thm:impossibility-delta} and Corollary \ref{cor:2_factor_cycle} imply that $\mathbb{P}(E) \le \delta n\cdot n^{-\Omega(\log n)} = n^{-\Omega(\log n)}$. 
\end{proof}


\section*{Acknowledgement}
J. Gaudio is supported in part by an NSF CAREER award CCF-2440539. J. Xu is supported in part by an NSF CAREER award CCF-2144593. 

\bibliography{references}
\bibliographystyle{plain}

\end{document}